\documentclass[12pt]{amsart}
\usepackage{leftidx}
\usepackage{mathrsfs}
\usepackage{amssymb}

\newcommand{\Sym}{\mathrm{Sym}}
\newcommand{\Tr}{\mathrm{Tr}}
\newcommand{\tr}{\mathrm{tr}}

\newtheorem{definition}{Definition}
\newtheorem{theorem}{Theorem}
\newtheorem{lemma}{Lemma}
\newtheorem{proposition}{Proposition}

\numberwithin{equation}{section}
\numberwithin{theorem}{section}
\numberwithin{definition}{section}
\numberwithin{lemma}{section}
\numberwithin{proposition}{section}

\allowdisplaybreaks

\begin{document}
\title{The Kohnen plus space for Hilbert-Siegel modular forms}
\author{Ren-He Su}
\address{Graduate school of mathematics, Kyoto University, Kitashirakawa, Kyoto, 606-8502, Japan}
\email{ru-su@math.kyoto-u.ac.jp}
\maketitle

\begin{abstract}
The Kohnen plus space, roughly speaking, is a space consisting of modular forms of half integral weight with some property in Fourier coefficients.
For example, the $n$-th coefficient of a normal modular form of weight $k+1/2$ in the plus space is $0$ unless $(-1)^kn$ is congruent to some square modulo $4$.
The concept of plus space was initially introduced by Kohnen in 1980.
Eichler and Zagier showed that the plus space is isomorphic to the space of Jacobi forms in the one variable case.
Later, Ibukiyama generalized these results to the cases for Siegel modular forms in 1992.
Also, Hiraga and Ikeda generalized these results to the cases for Hilbert modular forms in 2013.
In this paper, we continue to consider the case of Hilbert-Siegel modular forms.
An analogue of the previous results will be given.

\end{abstract}

\section*{Notations}
For any complex number $z\in\mathbb{C},$ put $\mathbf{e}(z)=e^{2\pi\sqrt{-1}z}$.
When $R$ is a ring and $m$ is a positive integer, $M_m(R)$ is the set consisting of all $m\times m$ matrices with entries in $R$ and $\Sym_m(R)$ consists of symmetric matrices in $M_m(R)$.
If $F$ is a global field with ring of integers $\mathfrak{o},$ a half-integral symmetric matrix in $M_m(F)$ is a matrix consisting of entries in $\frac{1}{2}\mathfrak{o}$ and in particular diagonal entries in $\mathfrak{o}.$
For a positive integer $n$ and some arbitrary ordered $n$-tuple $\alpha$, when there is no special remark, $\alpha_i$ automatically stands for the $i$-th component of $\alpha$ for $1\leq i\leq n.$
If $A$ is an $n$-tuple of matrices in $M_m(\mathbb{R}),$ the notation $A\geq0$ means that $A_i$ is positive semi-definite for $1\leq i\leq n$.
Similarly, the notation $A>0$ means that $A_i$ is positive definite for $1\leq i\leq n$.

\section*{Acknowledgements}
The author would like to sincerely show his gratitude to Prof. Ikeda for his very kind and useful advisements.
This paper could not be completed without his instruction.
The author also would like to dedicate this paper, his second thesis, to his family, friends and Prof. Ikeda.

\section{Introduction}
\label{Introduction}
Before introducing the main results, let us define the Hilbert-Siegel Jacobi forms and the plus space for Hilbert-Siegel modular forms.
\par
Let $F\neq\mathbb{Q}$ be a totally real field of degree $n$ over $\mathbb{Q}$ with integer ring $\mathfrak{o}$ and different $\mathfrak{d}$ over $\mathbb{Q}$.
We denote the $n$ real embeddings of $F$ by $\iota_i$ for $1\leq i\leq n$.
For $\xi\in F,$ $\iota_i(\xi)$ will sometimes simply be denoted by $\xi_i$ or $\xi_{\infty_i}$.
An element $\xi\in F$ will be considered as a real $n$-tuple.
\par
Also, let us fix a positive integer $m>1$.
The Siegel upper half-space of genus $m$ is defined by
\[
\mathfrak{h}_m=\{X+\sqrt{-1}Y\in M_m(\mathbb{C})|X, Y\in \Sym_m(\mathbb{R}), Y>0\}
\]
where as usual, $\Sym_m$ is the set of symmetric $m\times m$ matrices and $Y>0$ means $Y$ is positive definite.
The set $\mathfrak{h}_m^n$ consists of all $n$-tuples whose components are in $\mathfrak{h}_m$.
Similarly, $(\mathbb{C}^m)^n$ consists of all $n$-tuples having column vectors in $\mathbb{C}^m$ as components.
Note that any vector in this paper will be considered as a column vector.
\par
For any ring $R,$ the symplectic group of size $2m$ is defined by
\[
Sp_m(R)=\left\{g\in M_{2m}(R)\bigg|g\begin{pmatrix}
0_m&-I_m \\ I_m&0_m
\end{pmatrix}{^t\!g}=\begin{pmatrix}
0_m&-I_m \\ I_m&0_m
\end{pmatrix}
\right\}.
\]
It is well-known that $Sp_m(\mathbb{R})$ acts on $\mathfrak{h}_m\times\mathbb{C}^m$ by
\[
g(z,w)=(gz,^t\!(cz+d)^{-1}w)=((az+b)(cz+d)^{-1},^t\!(cz+d)^{-1}w)
\]
for $(z,w)\in\mathfrak{h}_m\times\mathbb{C}^m$ and
\[
g=\begin{pmatrix}a&b\\c&d\end{pmatrix}\in Sp_m(R)\quad(a,b,c,d\in M_n(\mathbb{R})).
\]
In the same way, if we consider $Sp_m(F)$ as an subset of $Sp_m(\mathbb{R})^n,$ then $Sp_m(F)$ acts on $\mathfrak{h}_m^n\times(\mathbb{C}^m)^n$ componentwisely. 
\par
Two important congruence subgroups of $Sp_m(F)$ are $\Gamma=\Gamma_0(1)$ and $\Gamma_0(4)$.
Their definitions are
\begin{equation}\label{defofgamma}
\Gamma_0(1)=\left\{
\begin{pmatrix}a&b\\c&d\end{pmatrix}\in Sp_m(F)\bigg|
a,c\in M_m(\mathfrak{0}), b\in M_m(\mathfrak{d}^{-1}), c\in M_m(\mathfrak{d})
\right\}
\end{equation}
and
\begin{equation}\label{defofgamma4}
\Gamma_0(4)=\left\{
\begin{pmatrix}a&b\\c&d\end{pmatrix}\in Sp_m(F)\bigg|
a,c\in M_m(\mathfrak{0}), b\in M_m(\mathfrak{d}^{-1}), c\in M_m(4\mathfrak{d})
\right\},
\end{equation}
respectively.
\par
Now we define the Hilbert-Siegel Jacobi forms.

\begin{definition}
\label{defofjacobiforms}
Let $G(z,w)$ be a holomorphic function on $\mathfrak{h}_m^n\times(\mathbb{C}^m)^n$ and $k=(k_i)_{1\leq i\leq n}$ be an $n$-tuple of positive integer.
If $G$ satisfies the following two conditions:\\[-0.3cm]
\item[(1)]$G(z,w+zx+y)=\mathbf{e}(-\Tr(^t\!xzx+2^t\!xw))G(z,w)$ for any $x\in\mathfrak{o}^m,y\in(\mathfrak{d}^{-1})^m$
\item[(2)]$G(\gamma(z,w))=N(\det(cz+d))^k\mathbf{e}(\Tr(^t\!w(cz+d)^{-1}cw))G(z,w)\left(\gamma=\begin{pmatrix}a&b\\c&d\end{pmatrix}\in\Gamma\right),$
then G is called a Jacobi form of weight $k$ and index $1$ or simply of weight $k$.
Here any addition and multiplication and determinant of $n$-tuples are simply calculated componentwisely and for any complex $n$-tuple $\tau=(\tau_i)_{1\leq i\leq n},$ we put
\[
\Tr(\tau)=\sum_{i=1}^n\tau_i,\quad N(\tau)^k=\prod_{i=1}^n\tau_i^{k_i}.
\]
The space of all Jacobi forms of weight $k$ is denoted by $J_{k,1}.$
\end{definition}
Let $G$ be a Jacobi form of weight $k$.
We define
\[
(G|_{k,1}g)(z,w)=G(g(z,w))N(\det(cz+d))^{-k}\mathbf{e}(-\Tr(^t\!w(cz+d)^{-1}cw))
\]
for $g=\begin{pmatrix}a&b\\c&d\end{pmatrix}$.
It is not difficult to see that $G|_{k,1}g$ has a Fourier expansion
\[
G|_{k,1}g=\sum_{N,r}f_g(N,r)\mathbf{e}(\Tr(\tr(Nz)))\mathbf{e}(\Tr(^t\!rz))
\]
where $N$ and $r$ run over certain lattice in $\Sym_m(F)$ and $F^m$, respectively, and that $f_g(N,r)=0$ unless $\begin{pmatrix}N&r/2\\^t\!r/2&1\end{pmatrix}\geq0$ by K\"{o}cher's principle.
Here $\tr$ denotes the usual matrix trace.
In particular, we let $f(N,r)=f_{I_{2m}}(N,r)$.

\begin{definition}
With the above notations, if $G$ has the property that 
\[
f_g(N,r)=0\mbox{ unless }\begin{pmatrix}N&r/2\\^t\!r/2&1\end{pmatrix}>0
\]for any $g\in Sp_m(F)$, we call $G$ a Jacobi cusp form.
The subspace of $J_{k,1}$ consisting of all Jacobi cusp forms is denoted by $J_{k,1}^\mathrm{CUSP}.$
\end{definition}

\par
For any $\lambda\in\mathfrak{o}^m/(2\mathfrak{o})^m,$ the theta function $\theta_\lambda$ is a function on $\mathfrak{h}_m^n\times(\mathbb{C}^m)^n$ defined by
\begin{equation}
\label{defoftheta_lambda0}
\theta_\lambda(z,w)=\sum_{p\in\mathfrak{o}^m}
\mathbf{e}\left(\Tr\left(^t\!(p+\frac{\lambda}{2})z(p+\frac{\lambda}{2})+2\cdot^t\!(p+\frac{\lambda}{2})w)\right)\right).
\end{equation}
One easily see that $\theta_\lambda$ does not depend on the choice of $\lambda$.
There are $2^{nm}$ distinct such theta functions.
Now if $G$ is a Jacobi form of weight $k,$ it is well-known that for any $\lambda\in(\mathfrak{o}/2\mathfrak{o})^m$ there is a unique holomorphic function $G_\lambda$ on $\mathfrak{h}_m^n$ such that
\begin{equation}
\label{eq:G_lambda}
G(z,w)=\sum_{\lambda\in(\mathfrak{o}/2\mathfrak{o})^m}G_\lambda(z)\theta_\lambda(z,w).
\end{equation}
The above formula is called the theta expansion for $G$.
\par
Now consider the function $\theta(z)=\theta_0(z,0)$ on $\mathfrak{h}_m^n$.
It is actually a modular form with respect to $\Gamma_0(4)$ of weight $1/2$.
The factor of automorphy of half-integral weight is defined by
\[
\tilde{j}(\gamma,z)=\frac{\theta(\gamma z)}{\theta(z)}\mbox{ for }
\gamma\in\Gamma_0(4)\mbox{ and }z\in\mathfrak{h}_m^n.
\]
It is shown in \cite{Shimura:85} that
\[
\tilde{j}(\gamma,z)^4=N(\det(cz+d))^2\mbox{ if }\gamma=\begin{pmatrix}a&b\\c&d\end{pmatrix}\in\Gamma_0(4).
\]
\par
We are now ready to define the plus space for Hilbert-Siegel modular forms.
Let $k=(k_1,...,k_n)$ be an $n$-tuple of positive integers.
For simplicity, here we only consider the case that $k_1\equiv k_2\equiv\cdots\equiv k_n$ (mod $2$).
The general case will be considered later in Section \ref{Automorphic forms on Sp_m(A)}.
The $n$-tuple $k$ is called even if its entries are even, or odd if its entries are odd.
If $k$ is parallel, i.e, if $k_1=k_2=\cdots=k_n,$ without any confusion, we denote the components of $k$ also by $k$.
We let $M_{k+1/2}(\Gamma_0(4))$ be the space of Hilbert-Siegel modular forms of weight $k+1/2$ with respect to $\Gamma_0(4),$  that is, $M_{k+1/2}(\Gamma_0(4))$ is the complex linear space of holomorphic functions $h$ on $\mathfrak{h}_m^n$ such that
\[
h(\gamma z)=J^{k+1/2}(\gamma,z)h(z)
\]
for any $\gamma=\begin{pmatrix}a&b\\c&d\end{pmatrix}\in\Gamma_0(4)$ where
\begin{equation}\label{defoftheautomorphicfactorparallel}
J^{k+1/2}(\gamma,z)=
\begin{cases}
\tilde{j}(\gamma,z)\prod_{j=1}^n\det(c_j z_j+d_j)^{k_j}\quad&\mbox{if }k\mbox{ is even,}\\
\tilde{j}(\gamma,z)^3\prod_{j=1}^n\det(c_j z_j+d_j)^{k_j-1}\quad&\mbox{if }k\mbox{ is odd.}
\end{cases}
\end{equation}
\par
Again, by K\"{o}cher's principle, we do not need the cusp condition for the definition of a modular form under the restrictions of $F$ and $m$.
A modular form $h\in M_{k+1/2}(\Gamma_0(4))$ has a Fourier expansion $h(z)=\sum_{T}c(T)\mathbf{e}(\Tr(\tr(Tz)))$ where in the summation $T$ runs over all positive semi-definite half-integral symmetric matrices and $c(T)\neq0$ only if $T\geq0$.
We call $h$ a cusp form if $h^4$ is a normal cusp Hilbert-Siegel modular form of weight $4k+2$.
The subspace of $M_{k+1/2}(\Gamma_0(4))$ consisting of the cusp forms is denoted by $S_{k+1/2}(\Gamma_0(4))$.

\begin{definition}
With the notations above, the plus space $M_{k+1/2}^+(\Gamma_0(4))$ is the subspace of $M_{k+1/2}(\Gamma_0(4))$ defined by
\begin{align*}M_{k+1/2}^+(\Gamma_0(4))&=\bigg\{
h(z)\in M_{k+1/2}(\Gamma_0(4))
\,\bigg|\,\\
c(T)&=0\mbox{ unless }T\equiv(-1)^k\lambda\cdot^t\!\lambda\mbox{ mod }4L_m^*\mbox{ for some }\lambda\in\mathfrak{o}^m
\bigg\}\end{align*}
where $L_m^*$ is the set of all $m\times m$ symmetric half-integral matrices.
Also, we let 
\[
S_{k+1/2}^+(\Gamma_0(4))=M_{k+1/2}^+(\Gamma_0(4))\cap S_{k+1/2}(\Gamma_0(4)).
\]
The space $S_{k+1/2}^+(\Gamma_0(4))$ is also called a plus space.
\end{definition}
The space just been defined is an analogue of which was initially brought up by Kohnen in \cite{Kohnen:80}. Also, the Siegel case and Hilbert case were established by Ibukiyama in \cite{Ibuki:92} and Hiraga and Ikeda in \cite{HiraIke:13}, respectively.
This is the reason why we only consider the case $F\neq\mathbb{Q}$ and $m>1$ in this paper, though similar result and proof apply to the previous cases if we add some adjustment for the cusp condition in the definitions.
\par
In this paper, as in \cite{HiraIke:13}, we will construct a Hecke operator $E^K$ on $M_{k+1/2}(\Gamma_0(4))$ and $S_{k+1/2}(\Gamma_0(4))$ such that the fixed subspaces of $E^K$ are the plus spaces.
We set
\[
\Gamma'=\left\{
\begin{pmatrix}a&b\\c&d\end{pmatrix}\in Sp_m(F)\bigg|
a,c\in M_m(\mathfrak{0}), b\in M_m((4\mathfrak{d})^{-1}), c\in M_m(4\mathfrak{d})
\right\}.
\]
\par
Let $\mathbb{A}$ be the adele ring of $F$.
We set a character of $\psi$ of $\mathbb{A}/F$ such that for any archimedean place $v$ of $F,$ the local character $\psi_v$ is given by $x\rightarrow\mathbf{e}((-1)^kx)$ for $x\in\mathbb{R}$ where $(-1)^k$ is $1$ if $k$ is even and $-1$ if $k$ is odd.
Fixing a non-archimedean place $v$ of $F$, the completion of $F$ with respect to $v$ is denoted by $F_v$.
Let $\widetilde{Sp_m(F_v)}$ be the metaplectic double covering of $Sp_m(F_v)$.
Also, for any subset $S\subset Sp_m(F_v),$ we denote its inverse image in $\widetilde{Sp_m(F_v)}$ by $\widetilde{S}$.
Let $\omega_{\psi_v}$ be the Weil representation of $\widetilde{Sp_m(F_v)}$ on the Schwartz space $\mathbb{S}(F_v^m)$.
The inner product for any two functions $\Phi_1$ and $\Phi_2$ in $\mathbb{S}(F_v^m)$ is defined by
\[
(\Phi_1,\Phi_2)=\int_{F_v^m}\Phi_1(X)\overline{\Phi_2(X)}dX.
\]
Here the Haar measure $dX$ on $F_v^m$ is normalized so that $Vol(\mathfrak{o}_v^m)=1$.
Now denote the characteristic function of $\mathfrak{o}^m$ by $\Phi_{0,v}$.
The local Hecke operator $E^K_v$ is defined by
\[
E^K_v(g)=\begin{cases}
|2|_v^m(\Phi_{0,v},\omega_{\psi_v}(g)\Phi_{0,v})\quad&\mbox{ if }g\in\widetilde{\Gamma'_v}\\
0\quad&\mbox{ otherwise.}
\end{cases}
\]
Let $\mathbb{A}_f=\prod'_{v<\infty}F_v$ be the finite part of $\mathbb{A}$.
The global Hecke operator $E^K$ is a function on the metaplectic double covering $\widetilde{Sp_m(\mathbb{A}_f)}$ of $Sp_m(\mathbb{A}_f)$ defined by
\[
E^K=\prod_{v<\infty}E^K_v.
\]
Note that $\widetilde{Sp_m(\mathbb{A}_f)}$ acts on the space of all automorphic forms lifted from the Hilbert modular forms of weight $k+1/2$ by the right translation $\rho$.
This induces a representation of $\widetilde{Sp_m(\mathbb{A}_f)}$ on the space of all Hilbert-Siegel modular forms of weight $k+1/2,$ which is also denoted by $\rho$.
For a Hecke operator $\varphi$ on $\widetilde{Sp_m(\mathbb{A}_f)}$ with some properties, $\varphi$ acts on $M_{k+1/2}(\Gamma_0(4))$ by
\[
\rho(\varphi)h(z)=\int_{\widetilde{Sp_m(\mathbb{A}_f)}/\{\pm1\}}(\rho(g)h)(z)\varphi(g)dg
\]
for any $h\in M_{k+1/2}(\Gamma_0(4))$ where $\{\pm1\}$ is the kernel of the canonical mapping $\widetilde{Sp_m(\mathbb{A}_f)}\rightarrow Sp_m(\mathbb{A}_f)$ and $dg$ is some normalized Haar measure on $\widetilde{Sp_m(\mathbb{A}_f)}/\{\pm1\}$.
The spaces $M_{k+1/2}(\Gamma_0(4))$ and $S_{k+1/2}(\Gamma_0(4))$ are invariant under the Hecke operator $E^K$.
Letting $M_{k+1/2}(\Gamma_0(4))^{E^K}$ and $S_{k+1/2}(\Gamma_0(4))^{E^K}$ denote the corresponding fixed subspaces, our first main result states that they are the plus spaces we defined above.
\begin{theorem}
We have
\[
M_{k+1/2}(\Gamma_0(4))^{E^K}=M_{k+1/2}^+(\Gamma_0(4))
\]
and
\[
S_{k+1/2}(\Gamma_0(4))^{E^K}=S_{k+1/2}^+(\Gamma_0(4)).
\]
\end{theorem}  
\par
Let $h(z)=\sum_{T}c(T)\mathbf{e}(\Tr(\tr(Tz)))\in M_{k+1/2}^+(\Gamma_0(4)).$
For any $\lambda\in\mathfrak{o}^m/(2\mathfrak{o}^m),$ we set
\begin{equation}
\label{eq:h_lambda}
h_\lambda(z)=\sum_{\begin{smallmatrix}-T\equiv\lambda^t\!\lambda\\\mathrm{mod}\,4L_m^*\end{smallmatrix}}c(T)\mathbf{e}\left(\frac{\Tr(\tr(Tz))}{4}\right).
\end{equation}
By the definition of the plus space, we have
\[
h(z)=\sum_{\lambda\in(\mathfrak{o}/2\mathfrak{o})^m}h_\lambda(4z).
\]
\par
Now we restrict us to the that case $k$ is odd unless $mn$ is even.
Our second main result states that $h_\lambda$ and $G_\lambda$ we defined before can be used to construct a Jacobi form and a modular form in the plus space, respectively.
\begin{theorem}
Let $k$ be an $n$-tuple of positive integers which is odd if $mn$ is odd.
For $h\in M_{k+1/2}^+(\Gamma_0(4))$ and $G\in J_{k+1,1},$ letting $h_\lambda$ and $G_\lambda$ be as in (\ref{eq:h_lambda}) and (\ref{eq:G_lambda}), respectively, we have
\[
\sum_{\lambda\in\mathfrak{o}^m/(2\mathfrak{o}^m)}h_\lambda(z)\theta_\lambda(z,w)\in J_{k+1,1}
\]
and
\[
\sum_{\lambda\in\mathfrak{o}^m/(2\mathfrak{o}^m))^m}G_\lambda(4z)\in M_{k+1/2}^+(\Gamma_0(4))
\]
where $\theta_\lambda$ are the theta functions defined in (\ref{defoftheta_lambda0}).
The two canonical mappings are the inverse of each other.
Thus we have
\[
M_{k+1/2}^+(\Gamma_0(4))\cong J_{k+1,1}
\]
as linear spaces over $\mathbb{C}.$
Moreover, we have
\[
S_{k+1/2}^+(\Gamma_0(4))\cong J_{k+1,1}^\mathrm{CUSP}.
\]
\end{theorem}
The case $F=\mathbb{Q}$ and $m=1$ for the theorem was given by Eichler and Zagier in \cite{EichZag:85}.
The Siegel case and Hilbert case were treated by Ibukiyama in \cite{Ibuki:92} and \cite{HiraIke:13}, respectively.
\par
Let us briefly state the contents in the rest of the paper.
First, we will introduce the Weil representation and give an important lemma about it in Section \ref{Weil Representation} and \ref{A key lemma}.
Next, we define the idempotents Hecke operators $e^K$ and $E^K$ in Section \ref{The idempotents $e^K$ and $E^K$}.
And we state some very brief facts we need about the archimedean places in Section \ref{The archimedean case}.
Using the results in the previous sections, we construct the automorphic forms of half integral weight in Section \ref{Automorphic forms on Sp_m(A)}.
Finally, we define the Kohnen plus space and the Jacobi forms and prove our two main theorems in Section \ref{The Kohnen plus space} and \ref{Relation to the Jacobi forms}.

\section{Weil Representations}
\label{Weil Representation}
Hereafter throughout the whole paper, $m>1$ is a fixed positive integer.
Let $F$ be a local field with characteristic $0$.
If $F$ is archimedean, we assume $F=\mathbb{R}$.
If $F$ a finite extension over $\mathbb{Q}_p,$ we let $\mathfrak{o}$ and $\mathfrak{p}$ denote its integer ring and prime ideal, respectively.
Moreover, we let $q$ and $\varpi$ be the order of $\mathfrak{o}/\mathfrak{p}$ and a prime element, respectively.
Now fix a non-trivial additive character $\psi:F\rightarrow\mathbb{C}^\times$.
If $F=\mathbb{R}$, we set $\psi(x)=\mathbf{e}(x)$ or $\mathbf{e}(-x)$.
In the non-archimedean case, the index of $\psi,$ which we denote by $c_\psi$, is the largest integer $c$ such that $\psi(\mathfrak{p}^{-c})=1$.
Also, we fix an element $\boldsymbol\delta$ of order $c_\psi$ if $F$ is non-archimedean.
If $F=\mathbb{R},$ we let $\boldsymbol\delta=1$.
Furthermore, the Haar measure $dx$ of $F$ is the unique one such that $\mathfrak{o}$ has volume $1$ if $F$ is non-archimedean or the usual Lebesgue measure otherwise.
The Haar measure $dX$ of $F^m$ is simply defined to be $\prod_{i}dx_i$ where we write $X={^t\!(x_1,x_2,...,x_m)}$.
\par
Now we denote the metaplectic double covering of $Sp_m(F)$ by $\widetilde{Sp_m(F)},$ that is,
\[
\widetilde{Sp_m(F)}=\{[g,\epsilon]|g\in Sp_m(F), \epsilon\in\{\pm 1\}\}
\] 
equipped with the group multiplication
\[
[g_1,\epsilon_1][g_2,\epsilon_2]=[g_1g_2,\epsilon_1\epsilon_2\boldsymbol{c}(g_1,g_2)].
\]
Here $\boldsymbol{c}(g_1,g_2)$ is Rao's 2-cocycle as in \cite{Rao:93}.
If $\mathbf{g}$ is an element in $\widetilde{Sp_m(F)},$ we set $\epsilon(\mathbf{g})\in\{\pm 1\}$ to be the latter component of $\mathbf{g}$.
\par
Some notations for elements in $\widetilde{Sp_m(F)}$ should be given for simplicity.
For any $g\in\widetilde{Sp_m(F)},$ let $[g]=[g,1]$.
Also, we let
\begin{align*}
\mathbf{u}^\sharp(B)&=\left[\begin{pmatrix}I_m&B\\0_m&I_m\end{pmatrix}\right],&\mathbf{u}^\flat(B)=\left[\begin{pmatrix}I_m&0_m\\B&I_m\end{pmatrix}\right],\quad&\mbox{for } B\in \Sym_m(F),\\
 \mathbf{m}(A)&=\left[\begin{pmatrix}A&0_m\\0_m&^t\!A^{-1}\end{pmatrix}\right],&\mathbf{w}_A=\left[\begin{pmatrix}0_m&-{^t}\!A^{-1}\\A&0_m\end{pmatrix}\right],\quad&\mbox{for } A\in GL_m(F).
\end{align*}
If $L$ is any subset of $Sp_m(F),$ $\widetilde{L}$ denotes the inverse image of $L$ in $\widetilde{Sp_m(F)}$.
\par
Set $\mathbb{S}(F^m)$ to be the space of Schwartz functions on $F^m$.
For any $\Phi\in\mathbb{S}(F^m),$ the Fourier transform of $\Phi$ is defined by
\begin{equation}
\widehat{\Phi}(X)=|\boldsymbol\delta|^{m/2}\int_{F^m}\Phi(Y)\psi(^tYX)dY.
\end{equation}
Note that $|\boldsymbol\delta|^{m/2}dX$ is the self-dual Haar measure for the Fourier transformation.
\par
It is known that for any $a\in F^\times,$ there is a constant $\alpha_\psi(a)$ such that
\[
\int_F\phi(x)\psi(ax^2)dx=\alpha_\psi(a)|2a|^{-1/2}\int_F\hat{\phi}(x)\psi(-\frac{x^2}{4a})dx
\]
where $\phi$ is a Schwartz function on $F$ and $\hat{\phi}$ is its Fourier transform defined in the similar manner as above.
The constant $\alpha_\psi(a)$ is called the Weil constant or the Weil index.
It satisfies $\alpha_\psi(a)^8=1$ and does not depend on $\phi$.
One can easily see that $\alpha_\psi(ab^2)=\alpha_\psi(a)$ for any $b\in F^\times$ and $\alpha_\psi(-a)=\overline{\alpha_\psi(a)}$.
\par
We now introduce the Weil representation of $\widetilde{Sp_m(F)}$ on $\mathbb{S}(F^m)$.
Let $\Phi$ be any Schwartz function in $\mathbb{S}(F^m),$ the Weil representation $\omega_\psi$ with respect to $\psi$ is given by
\begin{align*}
\omega_\psi(\mathbf{u}^\sharp(B))\Phi(X)&=\psi(^t\!XBX)\Phi(X),\\
\omega_\psi(\mathbf{m}(A))\Phi(X)&=\frac{\alpha_\psi(1)}{\alpha_\psi(\det A)}|\det A|^{1/2}\Phi(^t\!AX),\\
\omega_\psi(\mathbf{w}_{I_m})\Phi(X)&=\alpha_\psi(1)^{-m}|2|^{m/2}\widehat{\Phi}(-2X),
\end{align*}
where $B\in\Sym_m(F)$ and $A\in GL_m(F)$.
From these we get that
\begin{align*}
\omega_\psi(\mathbf{w}_C)\Phi(X)&=\frac{\alpha_\psi(1)^{1-m}}{\alpha_\psi(\det C)}|\det 2C^{-1}|^{1/2}\epsilon(\mathbf{w}_{I_m}\mathbf{m}(C))\widehat{\Phi}(-2C^{-1}X),\\
\omega_\psi(\mathbf{u}^\flat(S))\Phi(X)&=\frac{\alpha_\psi(1)^{1-2m}}{\alpha_\psi((-1)^m)}|4\boldsymbol\delta|^{m/2}\epsilon(\mathbf{w}_{I_m}\mathbf{m}(-I_m))\epsilon(\mathbf{w}_{-I_m}\mathbf{u}^\sharp(-S)\mathbf{w}_{I_m})
\\&\times\int_{F^m}\widehat{\Phi}(-2Y)\psi(-{^t}YSY+2^tYX)dY,
\end{align*}
where $S\in\Sym_m(F)$ and $C\in GL_m(F)$.
For any $\Phi_1, \Phi_2\in\mathbb{S}(F^m),$ the inner product of $\Phi_1$ and $\Phi_2$ is
\[
(\Phi_1,\Phi_2)=\int_{F^m}\Phi_1(X)\overline{\Phi_2(X)}dX.
\]
The Weil representation is unitary with respect to the inner product.
\par
From now in this section we suppose that $F$ is non-archimedean.
We write $\mathfrak{d}=\mathfrak{p}^{c_\psi}.$
As in the introduction, in the local case, we also let
\begin{equation}\label{defofgammaloc}
\Gamma=\Gamma_0(1)=\left\{
\begin{pmatrix}a&b\\c&d\end{pmatrix}\in Sp_m(F)\bigg|
a,c\in M_m(\mathfrak{o}), b\in M_m(\mathfrak{d}^{-1}), c\in M_m(\mathfrak{d})
\right\}
\end{equation}
and
\begin{equation}\label{defofgamma4loc}
\Gamma_0(4)=\left\{
\begin{pmatrix}a&b\\c&d\end{pmatrix}\in Sp_m(F)\bigg|
a,c\in M_m(\mathfrak{o}), b\in M_m(\mathfrak{d}^{-1}), c\in M_m(4\mathfrak{d})
\right\}.
\end{equation}
\par
In general, for any two fractional ideals $\mathfrak{b}$ and $\mathfrak{c}$ of $F$ such that $\mathfrak{bc}\subset\mathfrak{o,}$ we put
\[
\Gamma[\mathfrak{b},\mathfrak{c}]=\left\{
\begin{pmatrix}a&b\\c&d\end{pmatrix}\in Sp_m(F)\bigg|
a,c\in M_m(\mathfrak{o}), b\in M_m(\mathfrak{b}), c\in M_m(\mathfrak{c})
\right\}.
\]
\par
The following lemma is a well-known fact.

\begin{lemma}\label{generofgamma}
The compact open subgroup $\Gamma[\mathfrak{b},\mathfrak{c}]$ defined above is generated by the elements
$\begin{pmatrix}
A&0\\0&^t\!A^{-1}
\end{pmatrix},$
$\begin{pmatrix}
I_m&B\\0&I_m
\end{pmatrix}$
and
$\begin{pmatrix}
I_m&0\\C&I_m
\end{pmatrix}$
where $A\in GL_m(\mathfrak{o}),$ $B\in\Sym_m(\mathfrak{b})$ and $C\in \Sym_m(\mathfrak{c}).$
\end{lemma}
\par
Let $\Phi_0\in\mathbb{S}(F^m)$ be the characteristic function of $\mathfrak{o}^m$.
Using Lemma \ref{generofgamma}, after some calculation, we get the following lemma.

\begin{lemma}\label{defofthechar}
The restriction of $\omega_\psi$ to $\widetilde{\Gamma_0(4)}$ defines a genuine character $\varepsilon:\widetilde{\Gamma_0(4)}\rightarrow\mathbb{C}^\times$ by
\[
\omega_\psi(\mathbf{\gamma})\Phi_0=\varepsilon(\gamma)^{-1}\Phi_0\quad(\gamma\in\widetilde{\Gamma_0(4)}).
\]
\end{lemma}
\par
Let $e$ be the order of $2,$ that is, be the non-negative integer such that $|2|=q^{-e}$.
For $0\leq 1\leq e,$ set
\begin{align*}
\mathbb{S}^{(i)}&=\mathbb{S}((\mathfrak{p}^{-e})^m/(\mathfrak{p}^{-i})^m)\\
&=\{
f\in\mathbb{S}(F^m)\,\bigg|\,Supp(f)\subset(\mathfrak{p}^{-e})^m, f(X+Y)=f(X)\mbox{ for any }Y\in(\mathfrak{p}^{-i})^m
\}
\end{align*}
and
\[
\Gamma^{(i)}=\Gamma[\boldsymbol\delta^{-1}\mathfrak{p}^{2i},\boldsymbol\delta\mathfrak{o}].
\]
Hence we have $\Gamma^{(0)}=\Gamma_0(1)$.
For $0\leq i\leq e$ and any $\lambda\in\mathfrak{o}^m,$ we set $\Phi_\lambda^{(i)}\in\mathbb{S}(F^m)$ to be the characteristic function of $\lambda/2+(\mathfrak{p}^{-i})^m$.
Then
\[
\mathbb{S}^{(i)}=\bigoplus_{\lambda\in\mathfrak{o}^m/(\mathfrak{p}^{e-i})^m}\mathbb{C}\cdot\Phi_\lambda^{(i)}.
\]
So $\dim_\mathbb{C}\mathbb{S}^{(i)}=q^{m(e-i)}$.
It is worth mentioning that the Fourier transformation of $\Phi_\lambda^{(i)}$ is
\begin{equation}
\label{Fourtransofphilambda}
\widehat{\Phi_\lambda^{(i)}}(X)=|\boldsymbol\delta\varpi^{-2i}|^{m/2}\psi\left(\frac{^t\!X\lambda}{2}\right)\Phi_0(\boldsymbol\delta\varpi^{-i}X).
\end{equation}

\begin{proposition}
\label{invofOmega}
We restrict the Weil representation $\omega_\psi$ to $\widetilde{\Gamma^{(i)}}$ and denote this restricted representation by $\Omega_\psi^{(i)}$.
Then $\mathbb{S}^{(i)}$ is invariant with respect to $\Omega_\psi^{(i)}$.
\end{proposition}
\begin{proof}
Fix a vector $\lambda\in\mathfrak{o}^m$.
From Lemma \ref{generofgamma}, to show the invariance of $\mathbb{S}^{(i)},$ it suffices to show that $\Omega^{(i)}_\psi(\mathbf{m}(A))\Phi_\lambda^{(i)}$, $\Omega^{(i)}_\psi(\mathbf{u}^\sharp(\boldsymbol\delta^{-1}\varpi^{2i}B))\Phi_\lambda^{(i)}$
and $\Omega^{(i)}_\psi(\mathbf{u}^\flat(\boldsymbol\delta C))\Phi_\lambda^{(i)}$ all lie in $\mathbb{S}^{(i)}$ for any $A\in GL_m(\mathfrak{o})$ and any $B, C\in \Sym_m(\mathfrak{o}).$
The $\mathbf{m}(A)$ case is trivial.
The $\mathbf{u}^\sharp(\boldsymbol\delta^{-1}\varpi^{2i}B)$ case is also trivial, but it is worth mentioning that
\begin{equation}\label{decompofomegai}
\Omega^{(i)}_\psi(\mathbf{u}^\sharp(\boldsymbol\delta^{-1}\varpi^{2i}B))\Phi_\lambda^{(i)}=\psi\left(\varpi^{2i}\frac{{^t\!\lambda}B\lambda}{4\boldsymbol\delta}\right)\Phi_\lambda^{(i)}.
\end{equation}
Now we consider $\Omega^{(i)}_\psi(\mathbf{u}^\flat(\boldsymbol\delta C))\Phi_\lambda^{(i)}$.
Apparently, to get
\[
\Omega^{(i)}_\psi(\mathbf{u}^\flat(\boldsymbol\delta C))\Phi_\lambda^{(i)}\in\mathbb{S}^{(i)},
\]
it is sufficient to show that
\[
\omega_\psi(\mathbf{u}^\sharp(-C/\boldsymbol\delta))\omega_\psi(\mathbf{w}_{\boldsymbol\delta I_m})\Phi_\lambda^{(i)}\in\omega_\psi(\mathbf{w}_{\boldsymbol\delta I_m})\mathbb{S}^{(i)}.
\]
By the definition of the Weil representation, for any Schwartz function $\Phi$ in $\mathbb{S}(F^m)$, the function $\omega_\psi(\mathbf{w}_{\boldsymbol\delta I_m})\Phi(X)$ is a nonzero constant times of the Fourier transform $\widehat{\Phi}(-2X/\boldsymbol\delta)$.
But the dual lattices in $F^m$ associated to $(\mathfrak{p}^{-e})^m$ and $(\mathfrak{p}^{-i})^m$ with respect to $(X,Y)\mapsto\psi(-2{^t\!XY}/\boldsymbol\delta)$ are $\mathfrak{o}^m$ and $(\mathfrak{p}^{i-e})^m,$ respectively.
Thus
\begin{align*}
&\omega_\psi(\mathbf{w}_{\boldsymbol\delta I_m})\mathbb{S}^{(i)}=\mathbb{S}((\mathfrak{p}^{i-e})^m/\mathfrak{o}^m)\\
=&\{
f\in\mathbb{S}(F^m)\,\bigg|\,Supp(f)\subset(\mathfrak{p}^{i-e})^m, f(X+Y)=f(X)\mbox{ for any }Y\in\mathfrak{o}^m
\}
\end{align*}
(This also can be gotten from direct calculations).
But apparently $\omega_\psi(\mathbf{u}^\sharp(-C/\boldsymbol\delta))$ leaves this space fixed.
So we get that $\Omega^{(i)}_\psi(\mathbf{u}^\flat(\boldsymbol\delta C))\Phi_\lambda$ is in $\mathbb{S}^{(i)}$.
Here ends the proof for the invariance.
\end{proof}

By this proposition or calculating directly, we get an analogue of Lemma \ref{defofthechar}.
\begin{lemma}\label{defofthechar'}
The representation $\Omega_\psi^{(e)}$ defines a genuine character $\check{\varepsilon}$ of $\widetilde{\Gamma^{(e)}}$ by
\[
\Omega_\psi^{(e)}(\gamma)\Phi_0^{(e)}=\check{\varepsilon}(\gamma)^{-1}\Phi_0^{(e)}\quad(\gamma\in\widetilde{\Gamma^{(e)}}).
\]
\end{lemma}

Since $\omega_\psi(\mathbf{m}(2I_m))\Phi_0=\alpha_\psi(1)\overline{\alpha_\psi(2^m)}2^{m/2}\Phi_0^{(e)}$ and $\mathbf{m}(2I_m)^{-1}\widetilde{\Gamma^{(e)}}\mathbf{m}(2I_m)=\widetilde{\Gamma_0(4)},$ we have the following relation between $\varepsilon$ and $\check{\varepsilon}$:
\begin{equation}\label{relabettwochars}
\varepsilon(\mathbf{m}(2I_m)^{-1}\gamma\mathbf{m}(2I_m))=\check{\varepsilon}(\gamma)\quad(\gamma\in\widetilde{\Gamma^{(e)}}).
\end{equation}

The formula of the action of $\mathbf{u}^\flat$ on $\Phi_\lambda^{(i)}$ is useful in our paper.

\begin{lemma}\label{uflatonphi}
If $\lambda\in\mathfrak{o}^m$ and $S\in\Sym_m(\mathfrak{o}),$ we have the following formula
\begin{align*}
&\omega_\psi(\mathbf{u}^\flat(\mathbf{\boldsymbol\delta S}))\Phi_\lambda^{(i)}\\
=&\epsilon_S\sum_{\mu\in\mathfrak{o}^m/(\mathfrak{p}^{e-i})^m}\int_{\mathfrak{o}^m}\psi\left(\varpi^i\frac{^tY\lambda}{2\boldsymbol\delta}-\varpi^{2i}\frac{^tYSY}{4\boldsymbol\delta}-\varpi^i\frac{^tY\mu}{2\boldsymbol\delta}\right)dY\Phi_\mu^{(i)}\\
=&\epsilon_Sq^{m(i-e)}\sum_{\mu,\nu\in\mathfrak{o}^m/(\mathfrak{p}^{e-i})^m}\psi\left(\varpi^i\frac{^t\nu\lambda}{2\boldsymbol\delta}-\varpi^{2i}\frac{^t\nu S\nu}{4\boldsymbol\delta}-\varpi^i\frac{^t\nu\mu}{2\boldsymbol\delta}\right)\Phi_\mu^{(i)}
\end{align*}
where is $\epsilon_S$ is a fourth root of $1$ depending only on $S.$
\end{lemma}
\begin{proof}
This can be deduced by direct calculation.
Actually, by the definition of the Weil representation, we have
\begin{align*}
&\omega_\psi(\mathbf{u}^\flat(\boldsymbol\delta S))\Phi_\lambda^{(i)}(X)\\
=&\epsilon_S|4\boldsymbol\delta|^{m/2}\int_{F^m}\widehat{\Phi_\lambda^{(i)}}(-2Y)\psi(-\boldsymbol\delta\cdot{^tY}SY+2\cdot{^tY}X)dY\\
=&\epsilon_S|2\boldsymbol\delta\varpi^{-i}|^m\int_{F^m}\Phi_0(-2\boldsymbol\delta\varpi^{-i}Y)\psi(-^tY\lambda-\boldsymbol\delta\cdot{^tY}SY+2\cdot{^tY}X)dY\\
=&\epsilon_S\int_{\mathfrak{o}^m}\psi\left(\varpi^i\frac{^tY\lambda}{2\boldsymbol\delta}-\varpi^{2i}\frac{^tYSY}{4\boldsymbol\delta}-\varpi^i\frac{^tYX}{\boldsymbol\delta}\right)dY\\
=&\epsilon_S\sum_{\mu\in\mathfrak{o}^m/(\mathfrak{p}^{e-i})^m}\int_{\mathfrak{o}^m}\psi\left(\varpi^i\frac{^tY\lambda}{2\boldsymbol\delta}-\varpi^{2i}\frac{^tYSY}{4\boldsymbol\delta}-\varpi^i\frac{^tY\mu}{2\boldsymbol\delta}\right)dY\Phi_\mu^{(i)}(X)\\
=&\epsilon_Sq^{m(i-e)}\sum_{\mu,\nu\in\mathfrak{o}^m/(\mathfrak{o}^{e-i})^m}\psi\left(\varpi^i\frac{^t\nu\lambda}{2\boldsymbol\delta}-\varpi^{2i}\frac{^t\nu S\nu}{4\boldsymbol\delta}-\varpi^i\frac{^t\nu\mu}{2\boldsymbol\delta}\right)\Phi_\mu^{(i)}(X)
\end{align*}
where
\[
\epsilon_S=\frac{\alpha_\psi(1)^{1-2m}}{\alpha_\psi((-1)^m)}\epsilon(\mathbf{w}_{I_m}\mathbf{m}(-I_m))\epsilon(\mathbf{w}_{-I_m}\mathbf{u}^\sharp(-\boldsymbol\delta S)\mathbf{w}_{I_m})
\]
is a fourth root of $1$ by the properties of the Weil constant.
Note that we used Proposition \ref{invofOmega} in the fourth equation and the fact that the formulas does not depend on the choices of $\mu$ and $\nu$ in the fourth and fifth equations.
\end{proof}

\begin{proposition}\label{irrofOmega}
Given $0\leq i\leq e$, with the same notations in Proposition \ref{invofOmega}, we consider $\Omega_\psi^{(i)}$ as a representation of $\widetilde{\Gamma^{(i)}}$ on $\mathbb{S}^{(i)}$.
Then $\Omega_\psi^{(i)}$ is irreducible.
\end{proposition}
\begin{proof}
Obviously, as $\lambda$ running over elements in $\mathfrak{o}^m/(\mathfrak{p}^{e-i})^m,$ the functions $\psi(\varpi^{2i}\cdot^t\!\lambda D\lambda/(4\boldsymbol\delta))$ of $D\in\Sym_m(\mathfrak{o})$ give $q^{m(e-i)}$ distinct characters of $\Sym_m(\mathfrak{o})$.
So by equation (\ref{decompofomegai}) and the linear independence of distinct characters, we have that if $\mathbb{S}'$ is an invariant subspace of $\mathbb{S}^{(i)},$ then $\mathbb{S}'$ must take the form of
\[
\mathbb{S}'=\bigoplus_{\lambda\in S}\mathbb{C}\cdot\Phi_\lambda
\]
where $S$ is a subset of $\mathfrak{o}^m/(\mathfrak{p}^{e-i})^m$.
To get $\mathbb{S}'=\mathbb{S}^{(i)},$ it suffices to show that for a fixed $\lambda\in S$ and an arbitrarily chosen $\kappa\in\mathfrak{o}^m,$ there exists some $D\in\Sym_m(\mathfrak{o})$ such that
\[
\left(\Omega^{(i)}_\psi(\mathbf{u}^\flat(\boldsymbol\delta D)\Phi_\lambda^{(i)},\Phi_\kappa^{(i)}\right)\neq0.
\]
Fix one $\kappa$.
Say there are exactly $l$ components of $\lambda$ which are not congruent to the corresponding one of $\kappa$ modulo $\mathfrak{p}^{e-i}$.
Without loss of generality, we may assume $\lambda_j\not\equiv\kappa_j$ (mod $\mathfrak{p}^{e-i}$) exactly for $1\leq j\leq l$.
If we take a diagonal matrix $D'\in\Sym_m(\mathfrak{\mathfrak{o}})$ with diagonal entries $d_1,\dots, d_l\in\mathfrak{o}\backslash\{0\}$ and $d_{l+1}=\dots=d_m=0,$ then by Lemma \ref{uflatonphi}, we have
\begin{align*}
&\left(\Omega^{(i)}_\psi(\mathbf{u}^\flat(\boldsymbol\delta D')\Phi_\lambda^{(i)},\Phi_\kappa^{(i)}\right)\\
=&\epsilon_{D'}\int_{\mathfrak{o}^m}\psi\left(\varpi^i\frac{^tY\lambda}{2\boldsymbol\delta}-\varpi^{2i}\frac{\cdot^t\!YD'Y}{4\boldsymbol\delta}-\varpi^i\frac{^tY\kappa}{2\boldsymbol\delta}\right)dY\\
=&\epsilon_{D'}\prod_{j=1}^l\left[\int_\mathfrak{o}\psi\left(-\frac{d_j}{4\boldsymbol\delta}(\varpi^iy-\frac{\lambda_j-\kappa_j}{d_j})^2\right)dy\cdot\psi\left(\frac{(\lambda_j-\kappa_j)^2}{4\boldsymbol\delta d_j}\right)\right].
\end{align*}
Since $\lambda_j-\kappa_j\in\mathfrak{o}$ for any $1\leq j\leq l,$ for our purpose, we only need to show that for any $\tau\in\mathfrak{o},$ there exists some $d\in\mathfrak{o}\backslash\{0\}$ such that
\[
\int_\mathfrak{o}\psi\left(-\frac{d}{4\boldsymbol\delta}(\varpi^iy-\frac{\tau}{d})^2\right)dy\neq0.
\]
Consider the case $d=\tau^2u$ for some $u\in\mathfrak{o}^\times$.
Then it is reduced to show that there exists some unit $u$ such that
\begin{align*}
&\int_\mathfrak{o}\psi\left(-\frac{\tau^2u}{4\boldsymbol\delta}(\varpi^iy-\frac{1}{\tau u})^2\right)dy\\
=&\int_\mathfrak{o}\psi\left(-\frac{\varpi^{2i}\tau^2}{4\boldsymbol\delta u}(uy-\frac{1}{\varpi^i\tau})^2\right)dy\\
=&\int_\mathfrak{o}\psi\left(-\frac{\varpi^{2i}\tau^2}{4\boldsymbol\delta u}(y-\frac{1}{\varpi^i\tau})^2\right)dy
\end{align*}
is non-zero.
This simply follows from (2) of Lemma 2.10 in \cite{HiraIke:13}, so we get what we want to show.
\end{proof}


\section{A key lemma}
\label{A key lemma}
We use the same notations as in Section \ref{Weil Representation} and continue the assumption that $F$ is non-archimedean.
Let $\Omega_\psi=\Omega_\psi^{(0)}$ and $\Gamma=\Gamma_0(1)$.
The next lemma is essential in proving our main theorems.
\begin{lemma}
\label{mainlemma}
Let $\pi$ be a genuine representation of $\widetilde{Sp_m(F)}$ on a vector space $V$.
If there are $q^{em}$ vectors in $V,$ which are denoted by $h_\kappa$ for $\kappa\in\mathfrak{o}^m/(2\mathfrak{o})^m,$ such that $\pi$ has the properties that
\[
\pi(\mathbf{u}^\sharp(B/\boldsymbol\delta))h_\kappa=\psi\left(\frac{^t\!\kappa B\kappa}{4\boldsymbol\delta}\right)h_\kappa
\]
for any $B\in\Sym_m(\mathfrak{o}), \kappa\in\mathfrak{o}^m/(2\mathfrak{o})^m$ and
\[
\pi(\gamma)\left(\sum_{\kappa\in\mathfrak{o}^m/(2\mathfrak{o})^m}h_\kappa\right)=\check{\varepsilon}(\gamma)^{-1}\left(\sum_{\kappa\in\mathfrak{o}^m/(2\mathfrak{o})^m}h_\kappa\right)
\]
for any $\gamma\in\widetilde{\Gamma^{(e)}},$
then $(\pi|_{\widetilde{\Gamma}},\oplus_{\kappa}\mathbb{C}\cdot h_\kappa)$ forms a representation of $\widetilde{\Gamma}$ equivalent to the Weil representation $\Omega_\psi$ and $h_\kappa\mapsto\Phi_\kappa$ gives an intertwining map for it.
\end{lemma}
\begin{proof}
We use the induction to prove this lemma.
The spirit of the proof of Theorem 1 in \cite{Ibuki:92} will be applied.
For $0\leq i\leq e$ and $\kappa\in\mathfrak{o}^m,$ put
\[
h^{(e-i)}_\kappa=\sum_{\begin{smallmatrix}\lambda\equiv\kappa\,\mathrm{mod}\,(\mathfrak{p}^i)^m\\ \lambda\,\mathrm{mod}\,(2\mathfrak{o})^m\end{smallmatrix}}h_\lambda.
\]
This definition only depends on $\kappa$ mod $(\mathfrak{p}^i)^m$.
In particular, $h^{(0)}_\kappa=h_\kappa$ and $h^{(e)}_\kappa$ is the sum of all $h_\lambda$ for arbitrary $\kappa$.
By the assumption of the lemma, we already have that $\oplus_\kappa\mathbb{C}\cdot h^{(e)}_\kappa=\mathbb{C}\cdot(\sum_\lambda h_\lambda)$ is invariant under $\widetilde{\Gamma^{(e)}}$ and gives an representation equivalent to $\Omega_\psi^{(e)}$ under the map $h_\kappa^{(e)}\mapsto\Phi_\kappa^{(e)}$.
Now fix $0\leq i\leq e-1$ and assume that $\oplus_\kappa\mathbb{C}\cdot h_\kappa^{(e-i)}$ gives an representation of $\widetilde{\Gamma^{(e-i)}}$ equivalent to $\Omega_\psi^{(e-i)}$ and $h_\kappa^{(e-i)}\mapsto\Phi_\kappa^{(e-i)}$ forms an intertwining map.
We want to show that under this condition, the similar statement also holds for $\oplus_\kappa\mathbb{C}\cdot h_\kappa^{(e-i-1)}$.
Fix one $\kappa\in\mathfrak{o}^m$.
By Lemma \ref{generofgamma}, Lemma \ref{uflatonphi} and the assumption of the presenting lemma, it suffices to show that 
\begin{align}
&\pi(\mathbf{u}^\flat(\boldsymbol\delta S))h_\kappa^{(e-i-1)} \nonumber \\
=&\epsilon_Sq^{-m(i+1)}\sum_{\mu,\nu\in\mathfrak{o}^m/(\mathfrak{p}^{i+1})^m}\psi
\left(\frac{^t\nu\kappa}{\boldsymbol\delta\varpi^{i+1}}-\frac{^t\nu S\nu}{\boldsymbol\delta\varpi^{2(i+1)}}-\frac{^t\nu\mu}{\boldsymbol\delta\varpi^{i+1}}\right)h_\mu^{(e-i-1)}.\label{eq:key0}
\end{align}
Let $\Delta$ be the subgroup $Sym_m(\mathfrak{o}/\mathfrak{p}^{2i+1})$ consisting of all the diagonal matrices.
For $\lambda\in\mathfrak{o}^m$ and $D\in\Delta,$ one has
\[
\pi\left(\mathbf{u}^\sharp\left(\frac{4D}{\boldsymbol\delta\varpi^{2i+1}}\right)\right)h_\lambda^{(e-i)}
=\sum_{\begin{smallmatrix}\tau\equiv\lambda\,\mathrm{mod}\,(\mathfrak{p}^i)^m\\ \tau\,\mathrm{mod}\,(\mathfrak{p}^{i+1})^m\end{smallmatrix}}\psi\left(\frac{^t\tau D\tau}{\boldsymbol\delta\varpi^{2i+1}}\right)h_\tau^{(e-i-1)}.
\] 
(Note that the formula above does not depend on the choice of $D$ modulo $\mathfrak{p}^{2i+1},$ so the action is well-defined.)
For arbitrary $\tau\equiv\lambda$ mod $(\mathfrak{p}^i)^m,$ by Schur orthogonality relation for finite groups and the restriction of $i,$ it is easy to see that
\[
\sum_{D\in\Delta}\psi\left(\frac{^t\lambda D\lambda}{\boldsymbol\delta\varpi^{2i+1}}-\frac{^t\tau D\tau}{\boldsymbol\delta\varpi^{2i+1}}\right)=
\begin{cases}
q^{m(2i+1)}\quad&\mbox{if }\tau\equiv\lambda\mbox{ mod }(\mathfrak{p}^{i+1})^m,\\
0&\mbox{otherwise.}
\end{cases}
\]
Hence we get
\begin{equation}\label{eq:key1}
h_\kappa^{(e-i-1)}=q^{-m(2i+1)}\sum_{D\in\Delta}\psi\left(-\frac{^t\kappa D\kappa}{\boldsymbol\delta\varpi^{2i+1}}\right)\pi\left(\mathbf{u}^\sharp\left(\frac{4D}{\boldsymbol\delta\varpi^{2i+1}}\right)\right)h_\kappa^{(e-i)}.
\end{equation}
Thus
\begin{equation}\label{eq:key2}
\pi(\mathbf{u}^\flat(\boldsymbol\delta S))h_\kappa^{(e-i-1)}=q^{-m(2i+1)}\sum_{D\in\Delta}\psi\left(-\frac{^t\kappa D\kappa}{\boldsymbol\delta\varpi^{2i+1}}\right)\pi\left(\mathbf{u}^\sharp\left(\frac{4D}{\boldsymbol\delta\varpi^{2i+1}}\right)\boldsymbol\gamma_D\right)h_\kappa^{(e-i)}
\end{equation}
where
\[
\boldsymbol\gamma_D=\mathbf{u}^\sharp\left(-\frac{4D}{\boldsymbol\delta\varpi^{2i+1}}\right)\mathbf{u}^\flat(\boldsymbol\delta S)\mathbf{u}^\sharp\left(\frac{4D}{\boldsymbol\delta\varpi^{2i+1}}\right)\]
lies in $\widetilde{\Gamma^{(e-i)}}.$
If we denote the diagonal entries of $D$ by $d_1,\dots,d_m$ and let $\boldsymbol\gamma$ act on $\Phi_\kappa^{(e-i)},$ we get
\begin{align*}
&\Omega_\psi^{(e-i)}(\boldsymbol\gamma_D)\Phi_\kappa^{(e-i)}\\
=&\omega_\psi\left(\mathbf{u}^\sharp\left(-\frac{4D}{\boldsymbol\delta\varpi^{2i+1}}\right)\mathbf{u}^\flat(\boldsymbol\delta S)\mathbf{u}^\sharp\left(\frac{4D}{\boldsymbol\delta\varpi^{2i+1}}\right)\right)\Phi_\kappa^{(e-i)}\\
=&\sum_{\begin{smallmatrix}\lambda\equiv\kappa\,\mathrm{mod}\,(\mathfrak{p}^i)^m\\ \lambda\,\mathrm{mod}\,(\mathfrak{p}^{i+1})^m\end{smallmatrix}}\psi\left(\frac{^t\lambda D\lambda}{\boldsymbol\delta\varpi^{2i+1}}\right)\omega_\psi\left(\mathbf{u}^\sharp\left(-\frac{4D}{\boldsymbol\delta\varpi^{2i+1}}\right)\mathbf{u}^\flat(\boldsymbol\delta S)\right)\Phi_\lambda^{(e-i-1)}\\
=&\epsilon_Sq^{-m(i+1)}\sum_{\begin{smallmatrix}\lambda\equiv\kappa\,\mathrm{mod}\,(\mathfrak{p}^i)^m\\ \lambda\,\mathrm{mod}\,(\mathfrak{p}^{i+1})^m\\\mu,\nu\in\mathfrak{o}^m/(\mathfrak{p}^{i+1})^m\end{smallmatrix}}
\psi\left(\frac{^t\lambda D\lambda}{\boldsymbol\delta\varpi^{2i+1}}+\frac{^t\nu\lambda}{\boldsymbol\delta\varpi^{i+1}}-\frac{^t\nu S\nu}{\boldsymbol\delta\varpi^{2(i+1)}}-\frac{^t\nu\mu}{\boldsymbol\delta\varpi^{i+1}}\right)\\
&\times\omega_\psi\left(\mathbf{u}^\sharp\left(-\frac{4D}{\boldsymbol\delta\varpi^{2i+1}}\right)\right)\Phi_\mu^{(e-i-1)}\\
=&\epsilon_Sq^{-m(i+1)}\sum_{\begin{smallmatrix}\lambda\equiv\kappa\,\mathrm{mod}\,(\mathfrak{p}^i)^m\\ \lambda\,\mathrm{mod}\,(\mathfrak{p}^{i+1})^m\\\mu,\nu\in\mathfrak{o}^m/(\mathfrak{p}^{i+1})^m\end{smallmatrix}}
\psi\left(\frac{^t\lambda D\lambda}{\boldsymbol\delta\varpi^{2i+1}}+\frac{^t\nu\lambda}{\boldsymbol\delta\varpi^{i+1}}-\frac{^t\nu S\nu}{\boldsymbol\delta\varpi^{2(i+1)}}-\frac{^t\nu\mu}{\boldsymbol\delta\varpi^{i+1}}-\frac{^t\mu D\mu}{\boldsymbol\delta\varpi^{2i+1}}\right)\\
&\times\Phi_\mu^{(e-i-1)}
\end{align*}
Here for any $\nu\in\mathfrak{o}^m,$ we have
\begin{align*}
&\sum_{\begin{smallmatrix}\lambda\equiv\kappa\,\mathrm{mod}\,(\mathfrak{p}^i)^m\\ \lambda\,\mathrm{mod}\,(\mathfrak{p}^{i+1})^m\end{smallmatrix}}\psi\left(\frac{^t\lambda D\lambda}{\boldsymbol\delta\varpi^{2i+1}}+\frac{^t\nu\lambda}{\boldsymbol\delta\varpi^{i+1}}\right)\\
=&\psi\left(\frac{^t\kappa D\kappa}{\boldsymbol\delta\varpi^{2i+1}}+\frac{^t\nu\kappa}{\boldsymbol\delta\varpi^{i+1}}\right)\sum_{\begin{smallmatrix}\beta\in\mathfrak{o}^m/(\mathfrak{p}^1)^m\end{smallmatrix}}\psi\left(\frac{^t\beta D\beta}{\boldsymbol\delta\varpi}+\frac{^t\nu\beta}{\boldsymbol\delta\varpi}\right).\\
\end{align*}
Notice that $\beta\rightarrow\psi\left((^t\beta D\beta+^t\!v\beta)/(\boldsymbol\delta\varpi)\right)$ forms a character of $\beta\in\mathfrak{o}^m/(\mathfrak{p}^1)^m$.
We write $\nu\sim D$ if the character is trivial.
Them the equations above become
\begin{align*}
&\Omega_\psi^{(e-i)}(\boldsymbol\gamma_D)\Phi_\kappa^{(e-i)}\\
=&\epsilon_Sq^{-mi}\sum_{\begin{smallmatrix}\mu\in\mathfrak{o}^m/(\mathfrak{p}^{i+1})^m\\\nu\sim D\\\nu\,\mathrm{mod}\,(\mathfrak{p}^{i+1})^m\end{smallmatrix}}
\psi\left(\frac{^t\kappa D\kappa}{\boldsymbol\delta\varpi^{2i+1}}+\frac{^t\nu\kappa}{\boldsymbol\delta\varpi^{i+1}}-\frac{^t\nu S\nu}{\boldsymbol\delta\varpi^{2(i+1)}}-\frac{^t\nu\mu}{\boldsymbol\delta\varpi^{i+1}}-\frac{^t\mu D\mu}{\boldsymbol\delta\varpi^{2i+1}}\right)\\
&\times\Phi_\mu^{(e-i-1)}\\
=&\epsilon_Sq^{-mi}\sum_{\begin{smallmatrix}\mu\in\mathfrak{o}^m/(\mathfrak{p}^i)^m\\\nu\sim D\\\nu\,\mathrm{mod}\,(\mathfrak{p}^{i+1})^m\end{smallmatrix}}
\psi\left(\frac{^t\kappa D\kappa}{\boldsymbol\delta\varpi^{2i+1}}+\frac{^t\nu\kappa}{\boldsymbol\delta\varpi^{i+1}}-\frac{^t\nu S\nu}{\boldsymbol\delta\varpi^{2(i+1)}}-\frac{^t\nu\mu}{\boldsymbol\delta\varpi^{i+1}}-\frac{^t\mu D\mu}{\boldsymbol\delta\varpi^{2i+1}}\right)\\
&\times\Phi_\mu^{(e-i)}.
\end{align*}
Now applying this formula and the assumption of the induction back to equation (\ref{eq:key2}), we get
\begin{align*}
&\pi(\mathbf{u}^\flat(\boldsymbol\delta S))h_\kappa^{(e-i-1)}\\
=&\epsilon_Sq^{-m(3i+1)}\sum_{\begin{smallmatrix}\mu\in\mathfrak{o}^m/(\mathfrak{p}^i)^m\\D\in\Delta\end{smallmatrix}}\sum_{\begin{smallmatrix}\nu\sim D\\\nu\,\mathrm{mod}\,(\mathfrak{p}^{i+1})^m\end{smallmatrix}}\\
&\times\psi\left(\frac{^t\nu\kappa}{\boldsymbol\delta\varpi^{i+1}}-\frac{^t\nu S\nu}{\boldsymbol\delta\varpi^{2(i+1)}}-\frac{^t\nu\mu}{\boldsymbol\delta\varpi^{i+1}}-\frac{^t\mu D\mu}{\boldsymbol\delta\varpi^{2i+1}}\right)
\pi\left(\mathbf{u}^\sharp\left(\frac{4D}{\boldsymbol\delta\varpi^{2i+1}}\right)\right)h_\mu^{(e-i)}\\
=&\epsilon_Sq^{-m(3i+1)}\sum_{\begin{smallmatrix}\mu\in\mathfrak{o}^m/(\mathfrak{p}^i)^m\\D\in\Delta\end{smallmatrix}}\sum_{\begin{smallmatrix}\nu\sim D\\\nu\,\mathrm{mod}\,(\mathfrak{p}^{i+1})^m\end{smallmatrix}}\sum_{\begin{smallmatrix}\tau\equiv\mu\,\mathrm{mod}\,(\mathfrak{p}^i)^m\\ \tau\,\mathrm{mod}\,(\mathfrak{p}^{i+1})^m\end{smallmatrix}}\\
&\times\psi\left(\frac{^t\nu\kappa}{\boldsymbol\delta\varpi^{i+1}}-\frac{^t\nu S\nu}{\boldsymbol\delta\varpi^{2(i+1)}}-\frac{^t\nu\mu}{\boldsymbol\delta\varpi^{i+1}}-\frac{^t\mu D\mu}{\boldsymbol\delta\varpi^{2i+1}}+\frac{^t\tau D\tau}{\boldsymbol\delta\varpi^{2i+1}}\right)h_\tau^{(e-i-1)}\\
=&\epsilon_Sq^{-m(3i+1)}\sum_{\begin{smallmatrix}\mu\in\mathfrak{o}^m/(\mathfrak{p}^i)^m\\D\in\Delta\end{smallmatrix}}\sum_{\begin{smallmatrix}\nu\sim D\\\nu\,\mathrm{mod}\,(\mathfrak{p}^{i+1})^m\end{smallmatrix}}\sum_{\lambda\in\mathfrak{o}^m/(\mathfrak{p}^1)^m}\\
&\times\psi\left(\frac{^t\nu\kappa}{\boldsymbol\delta\varpi^{i+1}}-\frac{^t\nu S\nu}{\boldsymbol\delta\varpi^{2(i+1)}}-\frac{^t\nu\mu}{\boldsymbol\delta\varpi^{i+1}}+\frac{^t\lambda D\lambda}{\boldsymbol\delta\varpi}\right)h_{\mu+\varpi^i\lambda}^{(e-i-1)}.
\end{align*}
When $D$ runs over all the elements in $\Delta,$ one easily check that $\psi(^t\lambda D\lambda/(\boldsymbol\delta\varpi))$ give exactly all the $q^m$ distinct characters of $\lambda\in\mathfrak{o}^m/(\mathfrak{p}^1)^m$ for $q^{2i}$ times.
Hence we can associate one $\eta\in\mathfrak{o}^m/(\mathfrak{p}^1)^m$ to every $D\in\Delta$ such that $\psi(^t\lambda D\lambda/(\boldsymbol\delta\varpi))=\psi(^t\lambda\eta/(\boldsymbol\delta\varpi))$ for any $\lambda\in\mathfrak{o}^m/(\mathfrak{p}^1)^m$.
In this case, $\nu\sim D$ means that $-\nu\equiv\eta$ mod $(\mathfrak{p}^1)^m$.
So
\begin{align*}
&\pi(\mathbf{u}^\flat(\boldsymbol\delta S))h_\kappa^{(e-i-1)}\\
=&\epsilon_Sq^{-m(i+1)}\sum_{\begin{smallmatrix}\mu\in\mathfrak{o}^m/(\mathfrak{p}^i)^m\\\eta,\lambda\in\mathfrak{o}^m/(\mathfrak{p}^1)^m\end{smallmatrix}}\sum_{\theta\in\mathfrak{o}^m/(\mathfrak{p}^i)^m}\\
&\times\psi\left(\frac{^t(-\eta+\varpi\theta)\kappa}{\boldsymbol\delta\varpi^{i+1}}-\frac{^t(-\eta+\varpi\theta)S(-\eta+\varpi\theta)}{\boldsymbol\delta\varpi^{2(i+1)}}-\frac{^t(-\eta+\varpi\theta)\mu}{\boldsymbol\delta\varpi^{i+1}}+\frac{^t\lambda\eta}{\boldsymbol\delta\varpi}\right)h_{\mu+\varpi^i\lambda}^{(e-i-1)}\\
=&\epsilon_Sq^{-m(i+1)}\sum_{\begin{smallmatrix}\mu\in\mathfrak{o}^m/(\mathfrak{p}^i)^m\\\ \lambda\in\mathfrak{o}^m/(\mathfrak{p}^1)^m\end{smallmatrix}}\sum_{\nu\in\mathfrak{o}^m/(\mathfrak{p}^{i+1})^m}
\psi\left(\frac{^t\nu\kappa}{\boldsymbol\delta\varpi^{i+1}}-\frac{^t\nu S\nu}{\boldsymbol\delta\varpi^{2(i+1)}}-\frac{^t\nu\mu}{\boldsymbol\delta\varpi^{i+1}}-\frac{^t\lambda\nu}{\boldsymbol\delta\varpi}\right)h_{\mu+\varpi^i\lambda}^{(e-i-1)}\\
=&\epsilon_Sq^{-m(i+1)}\sum_{\tau,\nu\in\mathfrak{o}^m/(\mathfrak{p}^{i+1})^m}
\psi\left(\frac{^t\nu\kappa}{\boldsymbol\delta\varpi^{i+1}}-\frac{^t\nu S\nu}{\boldsymbol\delta\varpi^{2(i+1)}}-\frac{^t\nu\tau}{\boldsymbol\delta\varpi^{i+1}}\right)h_\tau^{(e-i-1)},
\end{align*}
which is equation (\ref{eq:key0}).
Thus it follows that $h_\lambda^{(e-i-1)}\mapsto\Phi_\psi^{(e-i-1)}$ gives an intertwining map between $(\pi|_{\widetilde{\Gamma^{(e-i-1)}}},\oplus_\lambda\mathbb{C}\cdot h_\lambda^{(e-i-1)})$ and $(\Omega_\psi^{(e-i-1)},\mathbb{S}^{(e-i-1)})$.
By the induction, what we wanted to show is proved.
\end{proof}


\section{The idempotents $e^K$ and $E^K$}
\label{The idempotents $e^K$ and $E^K$}
In this section $F$ is set to be non-archimedean and the same notations in Section \ref{Weil Representation} will be used.
Let $\varepsilon$ be the character of $\widetilde{\Gamma_0(4)}$ given by Lemma \ref{defofthechar}.
\begin{definition}\label{defofheckealgebra}
The Hecke algebra $\widetilde{\mathcal{H}}=\widetilde{\mathcal{H}}(\widetilde{\Gamma_0(4)}\backslash\widetilde{Sp_m(F)}/\widetilde{\Gamma_0(4)};\varepsilon)$ is the space consisting of all compactly supported genuine function $\vartheta$ on $\widetilde{Sp_m(F)}$ such $\vartheta(\gamma_1 g\gamma_2)=\varepsilon(\gamma_1\gamma_2)\vartheta(g)$ for any $\gamma_1$ and $\gamma_2$ in $\widetilde{\Gamma_0(4)}$.
The multiplication among $\widetilde{\mathcal{H}}$ is defined by
\[
(\vartheta_1\ast\vartheta_2)(g)=\int_{\widetilde{Sp_m(F)}/\{\pm1\}}\vartheta_1(gh^{-1})\vartheta_2(h)dh.
\]
Here the Haar measure $dh$ on $\widetilde{Sp_m(F)}$ is normalized so that the volume of $\widetilde{\Gamma}/\{\pm1\}$ is $1$.
\end{definition}
Now we introduce two matrix coefficients $e^K$ and $E^K$ on $\widetilde{\Gamma_0(1)}$ and $\mathbf{w}_{2\boldsymbol\delta I_m}\widetilde{\Gamma_0(1)}\mathbf{w}_{2\boldsymbol\delta I_m}^{-1}$ which are actually idempotents in $\widetilde{\mathcal{H}}$.
They play important roles in our paper.
\par
Recall that $\Gamma=\Gamma_0(1)$.
\begin{definition}
\label{defofeK}
The genuine function $e^K$ on $\widetilde{Sp_m(F)}$ is defined by
\[
e^K(\tilde{g})=
\begin{cases}
q^{me}(\Phi_0,\omega_\psi(\tilde{g})\Phi_0)&\mbox{if }g\in\widetilde{\Gamma},\\
0&\mbox{otherwise.}
\end{cases}
\]
Also, we put $E^K(\tilde{g})=e^K(\mathbf{w}_{2\boldsymbol\delta I_m}^{-1}\tilde{g}\mathbf{w}_{2\boldsymbol\delta I_m})$ for any $\tilde{g}\in\widetilde{Sp_m(F)}.$
\end{definition}
The support of $e^K$ is contained in $\widetilde{\Gamma}$ and the support of $E^K$ is contained in
\begin{align*}
&\mathbf{w}_{2\boldsymbol\delta I_m}\widetilde{\Gamma}\mathbf{w}_{2\boldsymbol\delta I_m}^{-1}\\
=&\left\{
\begin{pmatrix}a&b\\c&d\end{pmatrix}\in Sp_m(F)\bigg|
a,c\in M_m(\mathfrak{0}), b\in M_m(4^{-1}\mathfrak{d}^{-1}), c\in M_m(4\mathfrak{d})\right\}.
\end{align*}
Note that for $\tilde{g}\in\mathbf{w}_{2\boldsymbol\delta I_m}\widetilde{\Gamma}\mathbf{w}_{2\boldsymbol\delta I_m}^{-1},$ we have
\begin{align*}
E^K(\tilde{g})=&e^K(\mathbf{w}_{2\boldsymbol\delta I_m}^{-1}\tilde{g}\mathbf{w}_{2\boldsymbol\delta I_m})\\
=&q^{me}(\Phi_0,\omega_\psi(\mathbf{w}_{2\boldsymbol\delta I_m}^{-1}\tilde{g}\mathbf{w}_{2\boldsymbol\delta I_m})\Phi)\\
=&q^{me}(\omega_\psi(\mathbf{w}_{2\boldsymbol\delta I_m})\Phi_0,\omega_\psi(\tilde{g}\mathbf{w}_{2\boldsymbol\delta I_m})\Phi)\\
=&q^{me}(\Phi_0,\omega_\psi(\tilde{g})\Phi_0)
\end{align*}
by the unitarity of $\omega_\psi$ and the equation
\[
\omega_\psi(\mathbf{w}_{2\boldsymbol\delta I_m})\Phi_0=\frac{\alpha_\psi(1)^{1-m}}{\alpha_\psi((2\boldsymbol\delta)^m)}\epsilon(\mathbf{w}_{I_m}\mathbf{m}(2\boldsymbol\delta I_m))\Phi_0.
\]
Thus we see that both $e^K$ and $E^K$ are in $\widetilde{H}$.
The idempotence for $e^K$ and $E^K$ easily follow from Schur's orthogonality relation.
Obviously, we have $e^K, E^K\in\widetilde{\mathcal{H}}$.


\section{The archimedean case}
\label{The archimedean case}
We let $F=\mathbb{R}$ and $\psi(x)=\mathbf{e}(x)$.
The Weil constant $\alpha_\psi(x)$ is given by
\[
\alpha_\psi(x)=
\begin{cases}
\exp(\pi\sqrt{-1}/4)&\mbox{if }x>0\\
\exp(-\pi\sqrt{-1}/4)&\mbox{if }x<0.
\end{cases}
\]
\par
The real metaplectic group $\widetilde{Sp_m(R)}$ is the unique non-trivial topological double covering of $\widetilde{Sp_m(R)}$ with the multiplication defined in Section \ref{Weil Representation}.
It is known that there exists a unique factor of automorphy $\tilde{j}:\widetilde{Sp_m(R)}\times\mathfrak{h}_m\rightarrow\mathbb{C}$ for half-integral weight such that
\[
\tilde{j}\left(\left[\begin{pmatrix}
a&b\\c&d
\end{pmatrix},\xi\right],\tau\right)^2=\det(c\tau+d).
\]


\section{Automorphic forms on $\widetilde{Sp_m(\mathbb{A})}$}
\label{Automorphic forms on Sp_m(A)}
In this section, we let $F$ be a totally real field with degree $n>1$ over $\mathbb{Q}$.
The notations $\mathfrak{o},$ $\mathfrak{d}$ and the $\mathbb{A}=\mathbb{A}_F$ stand for the integer ring, the different and the adele ring of $F,$ respectively.
We want to take a brief look at the definition of the automorphic forms on $\widetilde{Sp_m(\mathbb{A})}$.
\par
As in Section \ref{Introduction}, the $n$ real embeddings of $F$ are denoted by $\infty_1,\dots,\infty_n$.
We let $\psi_1=\prod_{v\leq\infty}\psi_{1,v}$ be the non-trivial additive character of $\mathbb{A}/F$ such that $\psi_{1,\infty_j}(x)=\mathbf{e}(x)$ for all real places $\infty_j$.
So for any finite place $v,$ the index of $\psi_{1,v},$ which we denote by $c_{1,v},$ is the exponent of the corresponding prime ideal $\mathfrak{p}_v$ in the prime decomposition of $\mathfrak{d}$. 
For the sake of simplicity, from now, when the local case with respect to some place $v$ is being considered, we use the same notations given in Section \ref{Weil Representation} with a lower subscript $v$.
\par
We should give the definition of the global metaplectic group $\widetilde{Sp_m(\mathbb{A})}$.
If $v$ is a finite place of $F$ which is not even, there is a canonical splitting over $\Gamma_v$ in $\widetilde{\Gamma_v}$ where $\Gamma_v$ is defined by (\ref{defofgamma}).
The image of the canonical splitting is also denoted by $\Gamma_v$.
It is the stabilizer for $\Phi_{0,v}$ for almost all $v$.
The global metaplectic covering of $Sp_m(\mathbb{A}),$ which we denote by $\widetilde{Sp_m(\mathbb{A})},$ is the restricted direct product of $\widetilde{Sp_m(F_v)}$ with respect to $\{\Gamma_v\}$ divided by $\{(\xi_v)\in\prod_v\{\pm1\}\,|\,\prod_v\xi_v=1\}$.
Then the Weil representation $\omega_{\psi_1}$ of $\widetilde{Sp_m(\mathbb{A})}$ on the Schwartz space $\mathbb{S}(\mathbb{A}^m)$ is well-defined.
The group $Sp_m(F)$ can be embedded canonically into $\widetilde{Sp_m(\mathbb{A})},$ so we consider $Sp_m(F)$ as a subgroup of $\widetilde{Sp_m(\mathbb{A})}$ through this embedding.
As in the local case, for any $B\in\Sym_m(\mathbb{A})$ and $A\in GL_m(\mathbb{A}),$ we let
\begin{align*}
&\mathbf{u}^\sharp(B)=(\mathbf{u}^\sharp(S_v))_v,&\mathbf{u}^\flat(B)=(\mathbf{u}^\flat(S_v))_v,\\
&\mathbf{m}(A)=(\mathbf{m}(A_v))_v,&\mathbf{w}_A=(\mathbf{w}_{A_v})_v.
\end{align*}
If $S$ is a subset of $Sp_m(\mathbb{A}),$ we let $\widetilde{S}$ denote its inverse image in $\widetilde{Sp_m(\mathbb{A})}$.
\par
Let the group $\{\pm\mathbf{1}\}$ of order $2$ be the kernel of the canonical mapping $\widetilde{Sp_m(\mathbb{A})}\rightarrow Sp_m(\mathbb{A})$ where $\mathbf{1}$ is the identity element in $\widetilde{Sp_m(\mathbb{A})}$.
A function $f$ on $\widetilde{Sp_m(\mathbb{A})}$ is called a genuine function if $f(-\mathbf{1}\cdot\mathbf{g})=-f(\mathbf{g})$ for any $\mathbf{g}\in\widetilde{Sp_m(\mathbb{A})}$.
Let $(f_v)_{v\leq\infty}$ be a family of local genuine functions.
If $f_v(\mathbf{g}_v)=1$ for $g_v\in\Gamma_v$ for almost all finite non-even places $v,$ then the product $\prod_{v\leq\infty} f_v$ defined by $(\prod_{v\leq\infty}f_v)(\mathbf{g})=\prod_{v\leq\infty} f_v(\mathbf{g}_v)$ for if $\mathbf{g}=\prod_v\mathbf{g}_v$ ($\mathbf{g}_v\in\widetilde{Sp_m(F_v)}$) gives a genuine function on $\widetilde{Sp_m(\mathbb{A})}$.
Note that the decomposition $\mathbf{g}=\prod_v\mathbf{g}_v$ for $\mathbf{g}\in\widetilde{Sp_m(\mathbb{A})}$ is not unique.
But the function $\prod_v f_v$ is still well-defined.
\par
Let $\Gamma_f'$ be a compact open subgroup subgroup of $Sp_m(\mathbb{A}_f)$ and $\varepsilon'=\prod_{v<\infty}\varepsilon'_v:\widetilde{\Gamma_f'}\rightarrow\mathbb{C}^\times$ be a genuine character.
Fix an $n$-tuple $k=(k_i)_{i=1}^n\in\mathbb{Z}_{>1}^n$ of integers greater than $1$ and put $\Gamma'=Sp_m(F)\cap(\Gamma_f'\times Sp_m(\mathbb{R})^n)$.
We define a factor of automorphy $j_{\varepsilon'}^{k+1/2}(\gamma,z)$ for $\gamma\in\Gamma'$ and $z\in\mathfrak{h}_m^n$ by
\begin{equation}\label{defoftheautomorphicfactor}
j_{\varepsilon'}^{k+1/2}(\gamma,z)=\prod_{v<\infty}\varepsilon'_v([\gamma_v,1])\prod_{i=1}^n\tilde{j}([\iota_i(\gamma),1],z_i)^{2k_i+1}.
\end{equation}
With this factor of automorphy, we denote $M_{k+1/2}(\Gamma',\varepsilon')$ and $S_{k+1/2}(\Gamma',\varepsilon')$ the spaces of Hilbert-Siegel modular forms and cusp forms for $\Gamma'$ of weight $k+1/2$ with respect to the factor of automorphy $j_{\varepsilon'}^{k+1/2}(\gamma,z)$.
Thus if $h\in M_{k+1/2}(\Gamma',\varepsilon'),$ we have
\[
h(\gamma(z))=j_{\varepsilon'}^{k+1/2}(\gamma,z)h(z)
\]
for any $\gamma\in\Gamma'$ and $z\in\mathfrak{h}_m^n$.
We can associate $h$ to an automorphic form on $\widetilde{Sp_m(\mathbb{A})}$ which is genuine and left-invariant with respect to $Sp_m(F)$.
For any $\boldsymbol{g}\in\widetilde{Sp_m(\mathbb{A})},$ by the strong approximation theorem, 
there exist $\gamma\in Sp_m(F),$ $g_\infty\in\widetilde{Sp_m(\mathbb{R})^n}$ and $g_f\in\widetilde{\Gamma_f'}$ such that $\boldsymbol{g}=\gamma g_\infty g_h$.
Then we put
\[
\varphi_h(\boldsymbol{g})=h(g_\infty(\boldsymbol{i}))\varepsilon'(g_f)^{-1}\prod_{i=1}^n\tilde{j}(g_{\infty_i},\boldsymbol{i})^{-2k_i-1}
\]
where $\boldsymbol{i}=\sqrt{-1}(I_m,\dots,I_m)\in\mathfrak{h}_m^n$.
It is easy to see that $\varphi_h$ is well-defined and thus forms a genuine automorphic form on $Sp_m(F)\backslash\widetilde{Sp_m(\mathbb{A})}$.
\par
Using the notations above, we put
\[
\mathcal{A}_{k+1/2}(Sp_m(F)\backslash\widetilde{Sp_m(\mathbb{A})};\widetilde{\Gamma_f'},\varepsilon')=\{\varphi_h\,|\,h\in M_{k+1/2}(\Gamma',\varepsilon')\}
\]
and
\[
\mathcal{A}^\mathrm{CUSP}_{k+1/2}(Sp_m(F)\backslash\widetilde{Sp_m(\mathbb{A})};\widetilde{\Gamma_f'},\varepsilon')=\{\varphi_h\,|\,h\in S_{k+1/2}(\Gamma',\varepsilon')\}.
\]
Let $\varphi\in\mathcal{A}_{k+1/2}(Sp_m(F)\backslash\widetilde{Sp_m(\mathbb{A})};\widetilde{\Gamma_f'},\varepsilon')$.
For $z\in\mathfrak{h}_m^n,$ we can take some $g_\infty\in\widetilde{Sp_m(\mathbb{R})^n}$ such that $g_\infty(\boldsymbol{i})=z$.
If we set
\[
h_\varphi(z)=\varphi(g_\infty)\prod_{i=1}^n\tilde{j}(g_{\infty_i},\boldsymbol{i})^{2k_i+1},
\]
then $h_\varphi\in M_{k+1/2}(\Gamma',\varepsilon')$ and $h_{\varphi_{h'}}=h'$ for all $h'\in M_{k+1/2}(\Gamma',\varepsilon')$.
Hence we get an one-to-one correspondence between the two spaces $M_{k+1/2}(\Gamma',\varepsilon')$ and $\mathcal{A}_{k+1/2}(Sp_m(F)\backslash\widetilde{Sp_m(\mathbb{A})};\widetilde{\Gamma_f'},\varepsilon')$.
\par
We let
\[
\mathcal{A}_{k+1/2}(Sp_m(F)\backslash\widetilde{Sp_m(\mathbb{A})})=\bigcup_{(\Gamma'_f,\varepsilon)}\mathcal{A}_{k+1/2}(Sp_m(F)\backslash\widetilde{Sp_m(\mathbb{A})};\widetilde{\Gamma_f'},\varepsilon')
\]
and
\[
\mathcal{A}^\mathrm{CUSP}_{k+1/2}(Sp_m(F)\backslash\widetilde{Sp_m(\mathbb{A})})=\bigcup_{(\Gamma'_f,\varepsilon)}\mathcal{A}^\mathrm{CUSP}_{k+1/2}(Sp_m(F)\backslash\widetilde{Sp_m(\mathbb{A})};\widetilde{\Gamma_f'},\varepsilon')
\]
where in the unions $(\Gamma'_f,\varepsilon)$ runs over all pairs of compact open subgroups $\Gamma_f'$ of $Sp_m(\mathbb{A}_f)$ and genuine characters $\varepsilon'$ of $\Gamma_f'$.
The group $\widetilde{Sp_m(\mathbb{A}_f)}$ act on $\mathcal{A}_{k+1/2}(Sp_m(F)\backslash\widetilde{Sp_m(\mathbb{A})})$ and $\mathcal{A}^\mathrm{CUSP}_{k+1/2}(Sp_m(F)\backslash\widetilde{Sp_m(\mathbb{A})})$ by the right translation $\rho$.
These give corresponding actions of $\widetilde{Sp_m(\mathbb{A}_f)}$ on $\bigcup_{(\Gamma'_f,\varepsilon)}M_{k+1/2}(\Gamma',\varepsilon')$ and $\bigcup_{(\Gamma'_f,\varepsilon)}S_{k+1/2}(\Gamma',\varepsilon'),$ which we still denote by $\rho$.
Take a function $h$ in some $M_{k+1/2}(\Gamma',\varepsilon')$ with Fourier expansion $h(z)=\sum_{T\in\Sym_m(F)}c(T)\mathbf{e}(\Tr_{F/\mathbb{Q}}(\tr(Tz)))$.
Then for any $S\in\Sym_m(F_v)$ where $v$ is a finite place of $F,$ one can check that
\begin{equation}\label{formulaofrhousharp}
\rho(\mathbf{u}^\sharp(S))h(z)=\sum_{T\in\Sym_m(F)}c(T)\psi_{1,v}(\tr(TS))\mathbf{e}(\Tr_{F/\mathbb{Q}}(\tr(Tz)))
\end{equation}
Also, for any $A\in GL_m(F)$ with finite part $A_f$ and totally positive determinant $\det(A)\in F$, one can check that
\begin{equation}\label{formulaofmA}
\rho(\mathbf{m}(A_f))h(z)=\det(A)^{-k-1/2}h(A^{-1}z\cdot^t\!A^{-1})
\end{equation}
where $\det(A)^{-k-1/2}=\prod_j\det(A_{\infty_j})^{-k_j-1/2}$.
\par
Let $\eta\in\mathfrak{o}^\times$ be an unit such that $N_{F/\mathbb{Q}}(\eta)^m=(-1)^{m\sum_jk_j}$ and put $\psi(x)=\psi_1(\eta x)$ for any $x\in\mathbb{A}$.
By Lemma \ref{defofthechar}, there exists a genuine character $\varepsilon_v$ of $\widetilde{\Gamma_0(4)_v}$ constructed from the Weil representation $\omega_{\psi_v}$ for any finite place $v$ of $F$.
Put $\varepsilon=\prod_{v<\infty}\varepsilon_v,$ which is a character of $\widetilde{\Gamma_0(4)}_f=\widetilde{\prod_{v<\infty}'\Gamma_0(4)_v}$.
We can get a factor of automorphy $j_{\varepsilon}^{k+1/2}$ of half-integral weight from $\varepsilon$ and $\Gamma_0(4)_f$ by (\ref{defoftheautomorphicfactor}).
Note that $j_{\varepsilon}^{k+1/2}$ depends on the choice of $\eta\in\mathfrak{o}^\times$ and if $\eta$ does not satisfy the condition $N_{F/\mathbb{Q}}(\eta)^m=(-1)^{m\sum_jk_j},$ then one can check that $j_{\varepsilon}^{k+1/2}(\mathbf{m}(-I_m),z)$ is identically $-1$ so $M_{k+1/2}(\Gamma_0(4),\varepsilon)$ turns out to be the zero space.
If the components $k_j$ of $k$ are all congruent to each others modulo $2$ and $\eta=(-1)^{k_j}$ with an arbitrarily chosen $k_j$, it is known that the corresponding $j^{k+1/2}$ is the same with $J^{k+1/2}$ defined in (\ref{defoftheautomorphicfactorparallel}).
From now on, we consider the general case given in this paragraph.
We rewrite $j_{\varepsilon}^{k+1/2}$ by $J^{k+1/2}$ and put $M_{k+1/2}(\Gamma_0(4))=M_{k+1/2}(\Gamma_0(4)_f,\varepsilon)$ and $S_{k+1/2}(\Gamma_0(4))=S_{k+1/2}(\Gamma_0(4)_f,\varepsilon)$.
\par
For any finite place $v$ of $F,$ let $\widetilde{\mathcal{H}_v}=\widetilde{\mathcal{H}_v}(\widetilde{\Gamma_0(4)_v}\backslash\widetilde{Sp_m(F_v)}/\widetilde{\Gamma_0(4)_v};\varepsilon_v)$ denote the Hecke algebra with respect to $v$ as in Definition \ref{defofheckealgebra}.
Put $\widetilde{\mathcal{H}}=\otimes'_{v<\infty}\widetilde{\mathcal{H}_v}$ to be the restricted product of $\widetilde{\mathcal{H}_v}$ with respect to $\{\varepsilon_v\}_{v<\infty}$ where we set $\varepsilon_v=0$ outside $\widetilde{\Gamma_0(4)_v}$.
The Hecke algebra $\widetilde{\mathcal{H}}$ acts on $\mathcal{A}_{k+1/2}(Sp_m(F)\backslash\widetilde{Sp_m(\mathbb{A})})$ by
\[
\rho(\vartheta)\varphi(g)=\int_{\widetilde{Sp_m(F)}/\{\pm1\}}\vartheta(h)\varphi(gh)dh
\]
for $\vartheta\in\widetilde{\mathcal{H}}$ and $\vartheta\in\mathcal{A}_{k+1/2}(S_m(F)\backslash\widetilde{Sp_m(\mathbb{A})})$.
\par
Next, we let $e^K_v$ and $E^K_v$ in $\widetilde{\mathcal{H}_v}$ be the ones defined in Definition \ref{defofeK} for all finite place $v$.
Then we have that both $e^K=\prod_{v<\infty}e^K_v$ and $E^K=\prod_{v<\infty}E^K_v$ lie in $\widetilde{\mathcal{H}}$.
Let $\mathcal{A}_{k+1/2}(Sp_m(F)\backslash\widetilde{Sp_m(\mathbb{A})})^{E^K}$ and $\mathcal{A}^\mathrm{CUSP}_{k+1/2}(Sp_m(F)\backslash\widetilde{Sp_m(\mathbb{A})})^{E^K}$ be the subspaces fixed by $E^K,$ that is,
\[
\mathcal{A}_{k+1/2}(Sp_m(F)\backslash\widetilde{Sp_m(\mathbb{A})})^{E^K}=\{\varphi\in\mathcal{A}_{k+1/2}(Sp_m(F)\backslash\widetilde{Sp_m(\mathbb{A})})\,|\,\rho(E^K)\varphi=\varphi\}
\]
and
\[
\mathcal{A}^\mathrm{CUSP}_{k+1/2}(Sp_m(F)\backslash\widetilde{Sp_m(\mathbb{A})})^{E^K}=\{\varphi\in\mathcal{A}^\mathrm{CUSP}_{k+1/2}(Sp_m(F)\backslash\widetilde{Sp_m(\mathbb{A})})\,|\,\rho(E^K)\varphi=\varphi\}.
\]
Also, we give the corresponding fixed subspaces in $M_{k+1/2}(\Gamma_0(4))$ and $S_{k+1/2}(\Gamma_0(4))$.
\begin{definition}\label{defofE^Kfixedspace}
The subspaces of $M_{k+1/2}(\Gamma_0(4))$ and $S_{k+1/2}(\Gamma_0(4))$ corresponding to $\mathcal{A}_{k+1/2}(Sp_m(F)\backslash\widetilde{Sp_m(\mathbb{A})})^{E^K}$ and $\mathcal{A}^\mathrm{CUSP}_{k+1/2}(Sp_m(F)\backslash\widetilde{Sp_m(\mathbb{A})})^{E^K},$ respectively, are denoted by $M_{k+1/2}(\Gamma_0(4))^{E^K}$ and $S_{k+1/2}(\Gamma_0(4))^{E^K}$.
\end{definition}
We will give the definition of the plus spaces and show that they are exactly the $E^K$-fixed spaces in the last definition in the next section.


\section{The Kohnen plus space}
\label{The Kohnen plus space}
In this section we define of the plus spaces for the Hilbert-Siegel modular forms of half-integral weight and give our first main theorem.
The use of the notations in the last section will be continued.
\begin{definition}
For any symmetric matrix $T\in\Sym_m(F),$ we denote $T\equiv\square$ mod $4$ if there exists some vector $\lambda\in\mathfrak{o}^m$ such that $T-\lambda\cdot{^t\!\lambda}\in4L_m^*$ where $L_m^*\subset\Sym_m(F)$ consists of all $m\times m$ half-integral matrices in $\Sym_m(F)$.
\end{definition}
If $T\equiv\square$ mod $4,$ then apparently $T\in L_m^*$ and the corresponding $\lambda\in\mathfrak{o}^m$ is uniquely determined modulo $2\mathfrak{o}^m$.
\par
For any modular form $h\in M_{k+1/2}(\Gamma_0(4)),$ it can be written in the Fourier expansion
\[
h(z)=\sum_{\begin{smallmatrix}T\in L_m^*\\T>0\end{smallmatrix}}c(T)\mathbf{e}(\Tr_{F/\mathbb{Q}}(\tr(Tz))).
\]
Here $\Tr_{F/\mathbb{Q}}$ and $tr$ are the traces of $F/\mathbb{Q}$ and matrices, respectively, and $T>0$ means that the image of $T$ under every real embedding in $M_m(\mathbb{R})$ is positive semi-definite.
This follows from K\"{o}cher's principle.
\par
From now on, for simplicity, when the variable $z\in\mathfrak{h}_m^n$ is being considered, we let $q^T=\mathbf{e}(\Tr_{F/\mathbb{Q}}(\tr(Tz)))$.
\begin{definition}\label{defofplusspace}
The Kohnen plus space $M_{k+1/2}^+(\Gamma_0(4))$ and $S_{k+1/2}^+(\Gamma_0(4))$ are defined by
\begin{align*}
&M_{k+1/2}^+(\Gamma_0(4))\\
=&\left\{h(z)=\sum_Tc(T)q^T\in M_{k+1/2}(\Gamma_0(4))\,\bigg|\,c(T)=0\mbox{ unless }\eta^{-1}T\equiv\square\mbox{ mod }4\right\}
\end{align*}
and
\[
S_{k+1/2}^+(\Gamma_0(4))=M_{k+1/2}^+(\Gamma_0(4))\cap S_{k+1/2}(\Gamma_0(4)).
\]
\end{definition}
We shall show that the Kohnen plus spaces are actually the spaces fixed by $E^K$.
The following proposition and its proof are analogues of Proposition 13.4 and its proof in \cite{HiraIke:13}, respectively.

\begin{proposition}
\label{plustoE^K}
We have $M_{k+1/2}(\Gamma_0(4))^{E^K}\subset M_{k+1/2}^+(\Gamma_0(4))$.
\end{proposition}
\begin{proof}
Put $\hat{\mathfrak{o}}=\prod_{v<\infty}\mathfrak{o}_v$ and $\hat{\mathfrak{d}}=\mathfrak{d}\hat{\mathfrak{o}}$.
For each $v<\infty,$ we pick a certain fixed generator $\boldsymbol\delta_v\in\mathfrak{o}_v$ of $\mathfrak{d}_v$. 
Then $\boldsymbol\delta=(\boldsymbol\delta_v)_{v<\infty}\in\hat{\mathfrak{d}}$.
The quotient group $2^{-1}\hat{\mathfrak{o}}/\hat{\mathfrak{o}}$ is canonically isomorphic to $2^{-1}\mathfrak{o}/\mathfrak{o}$ by Chinese remainder theorem.
Let $\Gamma_f=\prod_{v<\infty}\Gamma_v$ and $\mathbb{S}(2^{-1}\hat{\mathfrak{o}}^m/\hat{\mathfrak{o}}^m)$ be the space of Schwartz functions on $2^{-1}\hat{\mathfrak{o}}^m/\hat{\mathfrak{o}}^m$.
By Proposition \ref{invofOmega} and \ref{irrofOmega}, the space $\mathbb{S}(2^{-1}\hat{\mathfrak{o}}^m/\hat{\mathfrak{o}}^m)$ gives an irreducible space for $\widetilde{\Gamma_f}$ through the Weil representation $\Omega_{\psi}=\otimes_{v<\infty}\Omega_{\psi_v}$. 
For $\lambda\in\hat{\mathfrak{o}}^m/(2\hat{\mathfrak{o}})^m,$ denote the characteristic function of $\lambda/2+\hat{\mathfrak{o}}^m$ by $\Phi_\lambda$.
The set of all such functions forms an orthonormal basis for $\mathbb{S}(2^{-1}\hat{\mathfrak{o}}^m/\hat{\mathfrak{o}}^m)$, which has the properties
\begin{align}
\Omega_\psi(e^K)\Phi_0&=\Phi_0,\\
\Omega_\psi(\mathbf{u}^\sharp(\boldsymbol\delta^{-1}S))\Phi_\lambda&=\psi(^t\!\lambda S\lambda/(4\boldsymbol\delta))\Phi_\lambda\quad\mbox{ for }S\in\Sym_m(\hat{\mathfrak{o}}),\\
\Omega_\psi(\mathbf{w}_{\boldsymbol\delta I_m})\Phi_0&=2^{-mn/2}\zeta_{\boldsymbol\delta}\sum_{\lambda\in\mathfrak{o}^m/(2\mathfrak{o})^m}\Phi_\lambda
\end{align}
where $\zeta_{\boldsymbol\delta}=\prod_{v<\infty}\alpha_{\psi_v}(1)^{1-m}\overline{\alpha_{\psi_v}(\boldsymbol\delta_v^m)}\epsilon_v(\mathbf{w}_{I_m}\mathbf{m}(\boldsymbol\delta I_m))$ is a fourth root of $1$ depending only on $\boldsymbol\delta$.
Now take some $h\in M_{k+1/2}(\Gamma_0(4))^{E^K}$.
We set
\[
h_0=2^{m\sum_jk_j}\zeta_{\boldsymbol\delta}\cdot\epsilon(\mathbf{w}_{\boldsymbol\delta I_m}\mathbf{w}_{-2\boldsymbol\delta I_m})^{-1}\rho(\mathbf{w}_{-2\boldsymbol\delta I_m})h
\]
where we set $\epsilon([g,\zeta])=\zeta$ for any $[g,\zeta]\in\widetilde{Sp_m(\mathbb{A}_f)}$.
By the definition of $E^K,$ we have $\rho(e^K)h_0=h_0$.
Let $\mathcal{V}$ be the $\mathbb{C}$-space spanned by $\{\rho(g)h_0\,|\,g\in\widetilde{\Gamma_f}\}$.
Since both $\Phi_0$ and $h_0$ are fixed by the matrix coefficient $e^K,$ there exists some intertwining map $i:\mathbb{S}(2^{-1}\hat{\mathfrak{o}}^m/\hat{\mathfrak{o}}^m)\rightarrow\mathcal{V}$ such that $i(\Phi_0)=h_0$.
For $\lambda\in\mathfrak{o}^m/(2\mathfrak{o}^m),$ denote $i(\Phi_\lambda)$ by $h_\lambda$.
Then we have
\begin{align}
\rho(e^K)h_0&=h_0,\\
\rho(\mathbf{u}^\sharp(\boldsymbol\delta^{-1}S))h_\lambda&=\psi(^t\!\lambda S\lambda/(4\boldsymbol\delta))h_\lambda\quad\mbox{ for }S\in\Sym_m(\hat{\mathfrak{o}}),\\
\rho(\mathbf{w}_{\boldsymbol\delta I_m})h_0&=2^{-mn/2}\zeta_{\boldsymbol\delta}\sum_{\lambda\in\mathfrak{o}^m/(2\mathfrak{o})^m}h_\lambda.
\end{align}
Let the Fourier expansion of $h_\lambda$ be $\sum_{T\in\Sym_m(F)}c_\lambda(T)q^{T/4},$ then for any $S\in\Sym_m(\hat{\mathfrak{o}}),$
\begin{align*}
&\sum_{T\in\Sym_m(F)}c_\lambda(T)\psi_1(\eta\cdot^t\!\lambda S\lambda/(4\boldsymbol\delta))q^{T/4}\\
=&\psi(^t\!\lambda S\lambda/(4\boldsymbol\delta))\sum_{T\in\Sym_m(F)}c_\lambda(T)q^{T/4}\\
=&\rho(\mathbf{u}^\sharp(\boldsymbol\delta^{-1}S))h_\lambda(z)\\
=&\sum_{T\in\Sym_m(F)}c_\lambda(T)\psi_1(\tr(TS)/(4\boldsymbol\delta))q^{T/4}
\end{align*}
where the latter equation is from (\ref{formulaofrhousharp}).
Thus we get that $c_\lambda(T)$ vanishes unless $\eta^{-1}T-\lambda\cdot^t\!\lambda\in4L_m^*$.
But
\begin{align*}
&\sum_{\lambda\in\mathfrak{o}^m/(2\mathfrak{o}^m)}h_\lambda(z)\\
=&2^{mn/2}\zeta_{\boldsymbol\delta}^{-1}\rho(\mathbf{w}_{\boldsymbol\delta I_m})h_0(z)\\
=&2^{m\sum_j(k_j+1/2)}\epsilon(\mathbf{w}_{\boldsymbol\delta I_m}\mathbf{w}_{-2\boldsymbol\delta I_m})^{-1}\rho(\mathbf{w}_{\boldsymbol\delta I_m}\mathbf{w}_{-2\boldsymbol\delta I_m})h(z)\\
=&2^{m\sum_j(k_j+1/2)}\rho(\mathbf{m}(2_fI_m))h(z)\\
=&h(z/4)
\end{align*}
by (\ref{formulaofmA}).
It follows that
\[
h(z)=\sum_{\lambda\in\mathfrak{o}^m/(2\mathfrak{o}^m)}\sum_{\eta^{-1}T-\lambda\cdot^t\!\lambda\in 4L_m^*}c_\lambda(T)q^T.
\]
Hence we have $h\in M_{k+1/2}^+(\Gamma_0(4))$.
\end{proof}

The converse of this proposition is an analogue of Proposition 13.3 in \cite{HiraIke:13}.
We introduce a different but simpler way to prove it.
Before the next proposition, let us put $\Gamma[4\mathfrak{d}^{-1},\mathfrak{d}]_f=\prod_{v<\infty}\Gamma[4\mathfrak{d}_v^{-1},\mathfrak{d}_v]$ and $\check{\varepsilon}=\prod_{v<\infty}\check{\varepsilon_v}$ where $\check{\varepsilon_v}$ is the character of $\widetilde{\Gamma[4\mathfrak{d}_v^{-1},\mathfrak{d}_v]}$ given in Lemma \ref{defofthechar'}.

\begin{proposition}
We have $M_{k+1/2}^+(\Gamma_0(4))\subset M_{k+1/2}(\Gamma_0(4))^{E^K}$.
\end{proposition}
\begin{proof}
Fix one $h(z)=\sum_{\begin{small}\eta^{-1}T\equiv\square\,\mathrm{mod}\,4\end{small}}c(T)q^T\in M_{k+1/2}^+(\Gamma_0(4))$.
We can write $h$ in the form $h(z)=\sum_{\lambda\in\mathfrak{o}^m/(2\mathfrak{o}^m)}h_\lambda(4z)$ where
\begin{equation}\label{defofh_lambda}
h_\lambda(z)=\sum_{\eta^{-1}T-^t\!\lambda\cdot\lambda\in4L_m^*}c(T)q^{T/4}.
\end{equation}
Let $\mathcal{V}$ be the $\mathbb{C}$-space spanned by $\{\rho(g)h_\lambda\,|\,g\in\widetilde{Sp_m(\mathbb{A}_f)}, \lambda\in\mathfrak{o}^m/(2\mathfrak{o}^m)\}$.
So $\mathcal{V}$ forms an invariant space of $\widetilde{Sp_m(\mathbb{A}_f)}$ by $\rho$.
Note that
\[
\rho(\mathbf{u}^\sharp(\boldsymbol\delta^{-1}S))h_\lambda=\psi(^t\!\lambda S\lambda/(4\boldsymbol\delta))h_\lambda
\]
for any $S\in\Sym_m(\hat{\mathfrak{o}})$ and $\lambda\in\mathfrak{o}^m/(2\mathfrak{o}^m)$ by (\ref{formulaofrhousharp}).
Also, for any $\boldsymbol\gamma\in\widetilde{\Gamma[4\mathfrak{d}^{-1},\mathfrak{d}]_f},$ we have
\begin{align*}
&\rho(\boldsymbol\gamma)\sum_{\lambda\in\mathfrak{o}^m/(2\mathfrak{o}^m)}h_\lambda\\\
=&2^{m\sum_j(k_j+1/2)}\rho({\boldsymbol\gamma}\mathbf{m}(2_fI_m))h(z)\\
=&2^{m\sum_j(k_j+1/2)}\varepsilon(\mathbf{m}(2_fI_m)^{-1}\boldsymbol\gamma\mathbf{m}(2_fI_m))^{-1}\rho(\mathbf{m}(2_fI_m))h(z)\\
=&\check{\varepsilon}(\boldsymbol\gamma)^{-1}\sum_{\lambda\in\mathfrak{o}^m/(2\mathfrak{o}^m)}h_\lambda
\end{align*}
by (\ref{relabettwochars}).
Now Lemma \ref{mainlemma} tells us that $(\rho|_{\widetilde{\Gamma_f}},\oplus_\lambda\mathbb{C}\cdot h_\lambda)$ is isomorphic to $(\Omega_\psi,\mathbb{S}(2^{-1}\hat{\mathfrak{o}}^m/\hat{\mathfrak{o}}^m))$ as representations of $\widetilde{\Gamma_f}$ under the intertwining map $h_\lambda\rightarrow\Phi_\lambda$.
This gives us that $\rho(e^K)h_0=h_0$.
Now since
\begin{align*}
&\rho(\mathbf{w}_{2\boldsymbol\delta I_m}^{-1})h_0\\
=&\zeta'\rho(\mathbf{m}(2_f^{-1}I_m)\mathbf{w}_{-\boldsymbol\delta I_m})h_0\\
=&2^{-mn/2}\zeta_{-\boldsymbol\delta}\zeta'\rho(\mathbf{m}(2_f^{-1}I_m))\sum_{\lambda\in\mathfrak{o}^m/(2\mathfrak{o})^m}h_\lambda\\
=&2^{\sum_jk_j}\zeta_{-\boldsymbol\delta}\zeta'\cdot h
\end{align*}
where
\[
\zeta'=\epsilon(\mathbf{w}_{2\boldsymbol\delta I_m}\mathbf{w}_{-2\boldsymbol\delta I_m})\epsilon(\mathbf{m}(2_f^{-1})\mathbf{w}_{-\boldsymbol\delta I_m})
\]
and
\[\zeta_{-\boldsymbol\delta}=\prod_{v<\infty}\alpha_{\psi_v}(1)^{1-m}\overline{\alpha_{\psi_v}((-\boldsymbol\delta_v)^m)}\epsilon_v(\mathbf{w}_{I_m}\mathbf{m}(-\boldsymbol\delta I_m)).
\]
So we get that $\rho(E^K)h=h$ by the definition of $E^K$.
\end{proof}

Our first main theorem follows from the two propositions.

\begin{theorem}\label{mainthm1}
The idempotent Hecke operator $E^K$ on $M_{k+1/2}(\Gamma_0(4))$ and $S_{k+1/2}(\Gamma_0(4))$ is just the projection to the plus spaces.
That is, we have
\[
M_{k+1/2}^+(\Gamma_0(4))=M_{k+1/2}(\Gamma_0(4))^{E^K}
\]
and
\[
S_{k+1/2}^+(\Gamma_0(4))=S_{k+1/2}(\Gamma_0(4))^{E^K}.
\]
\end{theorem}


\section{Relations to the Jacobi forms}
\label{Relation to the Jacobi forms}
In this section, we shall construct an isomorphism between the plus space and the space of Jacobi forms for certain restricted weights $k$.
But before that, let us give a brief introduction of the Jacobi forms.
For more detail, one can consult \cite{Ikeda:94} and \cite{RoRa:98}.
We use the same notations in the last section and assume $m\sum_jk_j\equiv mn$ mod $2$ and $\eta=-1$. 
\par
Let $G_{m+1}^J(F)$ be the subgroup of $Sp_{m+1}(F)$ consisting of all matrices whose first column is $^t\!(1,0,\dots,0)$.
If we embed $Sp_m(F)$ to $Sp_{m+1}(F)$ by
\[
\begin{pmatrix}
\Huge{A}&B\\C&D
\end{pmatrix}\mapsto
\left(
\begin{array}{c|c}
\begin{matrix}
1&\\&A
\end{matrix}&
\begin{matrix}
&\\&B
\end{matrix}\\
\hline
\begin{matrix}
&\\&C
\end{matrix}&
\begin{matrix}
1&\\&D
\end{matrix}
\end{array}\right)
\]
and define the Heisenberg group $H_m(F)$ by
\[
H_m(F)=\left\{
(X,Y,\kappa)=
\left(
\begin{array}{c|c}
\begin{matrix}
1&^t\!X\\&I_m
\end{matrix}&
\begin{matrix}
\kappa&^t\!Y\\Y&
\end{matrix}\\
\hline
\begin{matrix}
&\\&
\end{matrix}&
\begin{matrix}
1&\\-X&I_m
\end{matrix}
\end{array}\right)\in Sp_m(F)\,\bigg|\,
X,Y\in F^m, \kappa\in F
\right\},
\]
then it is easy to verify that $G_{m+1}^J(F)$ is the semi-direct product $Sp_m(F)\ltimes H_m(F)$.
The action of $G_{m+1}^J(F)$ on $\mathfrak{h}_m^n\times(\mathbb{C}^m)^n$ is given by
\[
\begin{pmatrix}a&b\\c&d\end{pmatrix}(X,Y,\kappa)(z,w)=((Az+B)(Cz+D)^{-1},^t\!(Cz+D)^{-1}(w+zX+Y))
\]
and it is transitive.
On the other hand, the adelic Heisenberg group $H_m(\mathbb{A})$ acts on the Schwartz space $\mathbb{S}(\mathbb{A}^m)$ by the Schr\"{o}dinger representation
\[
\pi_S(X,Y,\kappa)f(T)=\psi_1(\kappa+^t\!(2T+X)Y)f(T+X)
\]
where $\psi_1$ is the character on $\mathbb{A}/F$ defined in Section \ref{Automorphic forms on Sp_m(A)}.
Let $\widetilde{G_{m+1}^J(\mathbb{A})}=\widetilde{Sp_m(\mathbb{A})}\ltimes H_m(\mathbb{A})$ be the metaplectic double covering of the adelic Jacobi group $G_{m+1}^J(\mathbb{A})$.
The group $G_{m+1}^J(F)$ can be embedded into $\widetilde{G_{m+1}^J(\mathbb{A})}$.
Combining the Schr\"{o}dinger representation and the Weil representation of $\widetilde{Sp_m(\mathbb{A})},$ one can get the Schr\"{o}dinge-Weil representation $\pi_{SW}$ of $\widetilde{G_{m+1}^J(\mathbb{A})}$ on $\mathbb{S}(\mathbb{A}^m)$.
Now for any $\Phi\in\mathbb{S}(\mathbb{A}^m),$ the theta function $\Theta_\Phi$ associate to it is defined by
\[
\Theta_\Phi(\mathbf{g})=\sum_{\xi\in F^m}(\pi_{SW}(\mathbf{g})\Phi)(\xi)
\]
for any $\mathbf{g}\in\widetilde{G_{m+1}^J(\mathbb{A})}$.
Note that $\Theta_\Phi$ is a function on $\widetilde{G_{m+1}^J(\mathbb{A})}$ left-invariant under $G_{m+1}^J(F)$.
Now let us restrict us to the condition such that
\[
\Phi=\Phi_f\Phi_\infty\in\mathbb{S}(\mathbb{A}^m)
\]
where $\Phi_f\in\mathbb{S}(\mathbb{A}_f^m)$ and $\Phi_\infty(X_\infty)=\mathbf{e}(i\Tr(^t\!X_\infty\cdot X_\infty))$ for any $X_\infty\in(\mathbb{R}^m)^n$.
The space of all the theta functions constructed from such $\Phi$ is denoted by $\mathcal{A}_\theta(G_{m+1}^J(F)\backslash\widetilde{G_{m+1}^J(\mathbb{A})})$.
The group $\widetilde{G_{m+1}^J(\mathbb{A}_f)}$ acts on $\mathcal{A}_\theta(G_{m+1}^J(F)\backslash\widetilde{G_{m+1}^J(\mathbb{A})})$ by the right translation $\rho'$, which is obviously equivalent to the finite part of the  Schr\"{o}dinger-Weil representation $\pi_{SW}$.
By applying the genuine factor of automorphy $\tilde{j}_\theta$ on $\widetilde{G_{m+1}^J(\mathbb{R})}^n\times(\mathfrak{h}_m^n\times(\mathbb{C}^m)^n)$ given by
\begin{align*}
&\tilde{j}_\theta\left(\left[\begin{pmatrix}A&B\\C&D\end{pmatrix}\right](X,Y,\kappa),(z,w)\right)\\
=&\mathbf{e}(\Tr(^t\!w'(Cz+D)^{-1}Cw'-^t\!XzX-2\cdot^t\!Xw-^t\!XY-\kappa))\\
&\times\prod_{i=1}^n\tilde{j}\left(\left[\begin{pmatrix}A_{\infty_i}&B_{\infty_i}\\C_{\infty_i}&D_{\infty_i}\end{pmatrix}\right],z_i\right)
\end{align*}
where $w'=w+Xz+Y$ and $\tilde{j}$ is the factor of automorphy for half-integral weight given in Section \ref{The archimedean case}, any theta function $\Phi$ can be associated to a function on $\mathfrak{h}_m^n\times(\mathbb{C}^m)^n$ in the similar way with Section \ref{Automorphic forms on Sp_m(A)}, which is also called a theta function.
A quick calculations shows that, if for any $\lambda\in\mathfrak{o}^m/(2\mathfrak{o}^m),$ we put
\[
f_\lambda=\Phi_\lambda\Phi_\infty\in\mathbb{S}(\mathbb{A}^m),
\]
where $\Phi_\lambda$ is the characteristic function of $\lambda/2+\hat{\mathfrak{o}}^m$ and $\Phi_\infty$ is as above, the theta function on $\mathfrak{h}_m^n\times(\mathbb{C}^m)^n$ associated to $\Theta_{f_\lambda}$ is
\begin{equation}
\label{defoftheta_lambda}
\theta_\lambda(z,w)=\sum_{p\in\mathfrak{o}^m}\mathbf{e}\left(\Tr_{F/\mathbb{Q}}\left(^t\!(p+\frac{\lambda}{2})z(p+\frac{\lambda}{2})+2\cdot^t\!(p+\frac{\lambda}{2})w)\right)\right).
\end{equation}
The representation of $\widetilde{G_{m+1}^J(\mathbb{A}_f)}$ on the space of all the theta functions on  $\mathfrak{h}_m^n\times(\mathbb{C}^m)^n$ induced from this association is also denoted by $\rho'$.
Now consider the tensor product $\mathcal{A}_{k+1/2}(Sp_m(F)\backslash\widetilde{Sp_m(\mathbb{A})})\otimes_\mathbb{C}\mathcal{A}_\theta(G_{m+1}^J(F)\backslash\widetilde{G_{m+1}^J(\mathbb{A})})$.
The representation $\rho\otimes\rho',$ where $\rho$ is the right translation of $\widetilde{Sp_m(\mathbb{A}_f)}$ on $\mathcal{A}_{k+1/2}(Sp_m(F)\backslash\widetilde{Sp_m(\mathbb{A})}),$ forms a representation of $G_{m+1}^J(\mathbb{A}_f)$ on $\mathcal{A}_{k+1/2}(Sp_m(F)\backslash\widetilde{Sp_m(\mathbb{A})})\otimes_\mathbb{C}\mathcal{A}_\theta(G_{m+1}^J(F)\backslash\widetilde{G_{m+1}^J(\mathbb{A})})$.
Every element $\varphi$ in the tensor product space can be associated to exactly one function $G_\varphi$ on $\mathfrak{h}_m^n\times(\mathbb{C}^m)^n$ which is in the tensor product of the spaces of all Hilbert-Siegel modular forms of weight $k+1/2$ and all theta functions on $\mathfrak{h}_m^n\times(\mathbb{C}^m)^n$.
The function $G_\varphi$ is called a Jacobi form of weight $k+1$ if
\[
(\rho\otimes\rho')(\boldsymbol\gamma)\varphi=\varphi\mbox{ for any }\boldsymbol\gamma\in \Gamma_{m+1}^J=G_{m+1}^J(\mathbb{A}_f)\cap\left(\prod_{v<\infty}\Gamma_{m+1,v}\right).
\]
Here $\Gamma_{m+1,v}$ is the group defined by (\ref{defofgammaloc}) with $m$ replaced by $m+1$.
Notice that 
\[
\Gamma_{m+1}^J=\Gamma_f\ltimes\left(H_m(\mathbb{A}_f)\cap\left(\prod_{v<\infty}\Gamma_{m+1,v}\right)\right)
\]
where $\Gamma_f$ is the same with which in the proof of Proposition \ref{plustoE^K}.
It is known and easy to check that this definition for Jacobi forms coincides with Definition \ref{defofjacobiforms}.
The space of all Jacobi forms and all Jacobi cusp forms of weight $k+1$ are denoted by $J_{k+1,1}$ and $J_{k,1}^\mathrm{CUSP}$, respectively.
\par
Now we are ready to show our second main result.
For any $\lambda\in\mathfrak{o}^m/(2\mathfrak{o}^m),$ let $\theta_\lambda$ be the one as in (\ref{defoftheta_lambda}).
The Heisenberg group $H_m(\mathbb{A}_f)\cap(\prod_{v<\infty}\Gamma_{m+1,v})$ leaves all the theta functions $\theta_\lambda$ fixed.
Now assume $G$ is a Jacobi form of weight $k+1$.
There exist $2^{mn}$ uniquely determined holomorphic functions $G_\lambda$ on $\mathfrak{h}_m^n$ for $\lambda\in\mathfrak{o}^m/(2\mathfrak{o}^m)$ such that
\[
G(z,w)=\sum_{\lambda\in\mathfrak{o}^m/(2\mathfrak{o}^m)}G_\lambda(z)\theta_\lambda(z,w).
\]
It is known that the function $G_\lambda$ is a Hilbert-Siegel modular form of weight $k+1/2$ for some congruence subgroup of $Sp_m(F)$ for every $\lambda$.
And $G$ is a Jacobi cusp form if and only if every $G_\lambda$ is a cusp form.
The space $\oplus_\lambda\mathbb{C}\cdot\theta_\lambda$ forms a representation of $\widetilde{\Gamma_f}$ by $\rho'|_{\widetilde{\Gamma_f}}$.
This representation is isomorphic to the Weil representation $\Omega_{\psi_1}$ under the intertwining map $\theta_\lambda\mapsto\Phi_\lambda,$ which is irreducible by Property \ref{irrofOmega}.
Now since $(\theta_\lambda)_\lambda$ form an orthonormal basis of $\oplus_\lambda\mathbb{C}\cdot\theta_\lambda$ and the Weil representation $\Omega_\psi$ is unitary, the invariance of $G$ under $\Gamma_f$ implies that $(\rho|_{\widetilde{\Gamma_f}},\oplus_\lambda\mathbb{C}\cdot h_\lambda)$ forms a genuine representation of $\widetilde{\Gamma_f}$ which is isomorphic to $\overline{\Omega_{\psi_1}}=\Omega_{\overline{\psi_1}}=\Omega_\psi$ via the intertwining map $G_\lambda\mapsto\Phi_\lambda$.
Note that the intertwining map is unique up to scalar multiplication due to the irreducibility of $\Omega_\psi$.
So by the same argument as in the proof of Proposition \ref{plustoE^K}, we get
\[
\sum_{\lambda\in\mathfrak{o}^m/(2\mathfrak{o}^m)}G_\lambda(4z)\in M_{k+1/2}^+(\Gamma_0(4)).
\]
\par
Conversely, let $h(z)=\sum_\lambda h_\lambda(z/4)\in M_{k+1/2}^+(\Gamma_0(4))$ where $h_\lambda$ is given by (\ref{defofh_lambda}).
The space $\oplus_\lambda\mathbb{C}\cdot h_\lambda(z)$ forms an irreducible representation of $\widetilde{\Gamma_f}$ by $\rho|_{\widetilde{\Gamma_f}}$ which is isomorphic to $\Omega_\psi=\Omega_{\overline{\psi_1}}=\overline{\Omega_{\psi_1}}$ via $h_\lambda\mapsto\Phi_\lambda$.
Hence $\rho'|_{\widetilde{\Gamma_f}}=\overline{\rho|_{\widetilde{\Gamma_f}}}$.
Under this condition, we have that $\sum_\lambda h_\lambda(z)\theta_\lambda(z,w)$ is invariant under $\Gamma_f$ by $(\rho\otimes\rho')|_{\Gamma_f}$ according to the basic representation theory.
Also, $\sum_\lambda h_\lambda(z)\theta_\lambda(z,w)$ is fixed by the actions of $H_m(\mathbb{A}_f)\cap(\prod_{v<\infty}\Gamma_{m+1,v})$ since so are the theta functions $\theta_\lambda$, thus forms a Jacobi form.
\par
We conclude our results in the following theorem.
\begin{theorem}
For any Jacobi form $G=\sum_\lambda G_\lambda(z)\theta_\lambda(z,w)\in J_{k+1,1},$ we have
\[
\sum_{\lambda\in\mathfrak{o}^m/(2\mathfrak{o}^m)}G_\lambda(4z)\in M_{k+1/2}^+(\Gamma_0(4)).
\]
Conversely, for any Hilbert-Siegel modular form $h=\sum_\lambda(z)h_\lambda(4z)\in M_{k+1/2}^+(\Gamma_0(4)),$ we have
\[
\sum_{\lambda\in\mathfrak{o}^m/(2\mathfrak{o}^m)}h_\lambda(z)\theta_\lambda(z,w)\in J_{k+1,1}.
\]
The associations above are inverse to each other and thus give a canonical isomorphism between $M_{k+1/2}^+(\Gamma_0(4))$ (resp. $S_{k+1/2}^+(\Gamma_0(4))$) and $J_{k+1,1}$ (resp. $J_{k,1}^\mathrm{CUSP}$).

\end{theorem}


\end{document}